\documentclass[11pt, reqno]{amsart}
\usepackage[utf8]{inputenc}
\usepackage[english]{babel}
 
\usepackage[square,sort,comma,numbers]{natbib}

\usepackage[dvipsnames]{xcolor}

\usepackage{fancyhdr}
\newcommand{\nc}{\newcommand}
\nc{\dmo}{\DeclareMathOperator}
\usepackage{pinlabel}
\usepackage{graphicx}
\usepackage{tikz}
\usetikzlibrary{arrows}
\usetikzlibrary{matrix,arrows,decorations.pathmorphing}
\usepackage{fullpage} \usepackage{wrapfig}

\usepackage{float}

\usepackage{amsmath}
\usepackage{amssymb}
\usepackage{amsthm}

\numberwithin{equation}{section}
\numberwithin{figure}{section}

\usepackage{stmaryrd}
\usepackage[all]{xy}
\usepackage[mathcal]{eucal}
\usepackage{verbatim} 
\usepackage{fullpage} 
\usepackage{wrapfig}
\usepackage{lipsum}
\usepackage[all]{xy}
\theoremstyle{plain}
\newtheorem{thm}{Theorem}[section]

\newtheorem{lem}[thm]{Lemma}

\newtheorem{prop}[thm]{Proposition}
\newtheorem{cor}[thm]{Corollary}

\newtheorem{que}[thm]{Problem}
\theoremstyle{remark}

\newcommand{\discColor}{Black}
\newcommand{\curveColor}{Maroon}
\newcommand{\discSizeX}{2.1}

\theoremstyle{definition}
\newtheorem{defn}[thm]{Definition}
\newtheorem{rmk}[thm]{Remark}
\newtheorem{nota}[thm]{Notation}
\nc{\para}[1]{\medskip\noindent\textbf{#1.}}
\title{Section problems for configuration spaces of surfaces}
\author{Lei Chen}

\begin{document}
 \bibliographystyle{alpha}
\maketitle

\begin{abstract}
In this paper we give a close-to-sharp answer to the basic questions: When is there a continuous way to add a point to a configuration of $n$ ordered points on a surface $S$ of finite type so that all the points are still distinct? When this is possible, what are all the ways to do it?  More precisely, let PConf$_n(S)$ be the space of ordered $n$-tuple of distinct points in $S$.  Let $f_n(S): \text{PConf}_{n+1}(S) \to \text{PConf}_n(S)$ be the map given by $f_n(x_0,x_1,\ldots ,x_n):=(x_1,\ldots ,x_n)$. We classify all continuous sections of $f_n$ up to homotopy by proving the following.

1. If $S=\mathbb{R}^2$ and $n>3$, any section of $f_{n}(S)$ is either ``adding a point at infinity" or ``adding a point near $x_k$". (We define these two terms in Section 2.1; whether we can define ``adding a point near $x_k$" or ``adding a point at infinity" depends in a delicate way on properties of $S$. )

2. If $S=S^2$ a $2$-sphere and $n>4$, any section of $f_{n}(S)$ is ``adding a point near $x_k$"; if $S=S^2$ and $n=2$, the bundle $f_n(S)$ does not have a section. (We define this term in Section 3.2)

3. If $S=S_g$ a surface of genus $g>1$ and for $n>1$, we give an easy proof of \cite[Theorem 2]{MR1977999} that the bundle $f_{n}(S)$ does not have a section. 

\end{abstract}
\section{Introduction}
Let $M$ be a manifold. There is a natural geometric question: How can we continuously introduce a new point on $M$ for any collection of $n$ distinct points on $M$? We denote by PConf$_n(M)$ \emph{the pure configuration space} parametrizing ordered n-tuple of distinct points on $M$. Let $f_n(M): \text{PConf}_{n+1}(M) \to \text{PConf}_n(M)$ be the map given by $f_n(x_0,x_1,\ldots ,x_n):=(x_1,\ldots ,x_n)$. There is a natural action of permutation group $\Sigma_n$ on PConf$_n(M)$ by permuting the $n$ points. Permutation group $\Sigma_n$ acts on the fiber bundle $f_n(M)$ as well. Thus we get a new fiber bundle $F_n(M): \text{PConf}_{n+1}(M)/\Sigma_n \to  \text{PConf}_n(M)/\Sigma_n$, given by $F_n(x_0,\{x_1,\ldots ,x_n\}):=\{x_1,\ldots ,x_n\}$. The quotient $\text{PConf}_n(M)/\Sigma_n=:\text{Conf}_n(M)$ is called \emph{the configuration space} parametrizing unordered n-tuple of distinct points on $M$. In this article, we will study the existence and uniqueness of sections of  $f_n(M)$ and $F_n(M)$ when $M$ is a surface.

The study of sections of configuration spaces of open manifolds goes back to the work of McDuff and Segal \cite{MR0353298} \cite{MR0358766}. They introduce a point ``at infinity", which allows them to prove homological stability for configuration spaces. For closed manifolds, the possibility of adding a point depends on the topology of the manifold. For a manifold $M$ with a nowhere vanishing vector field, Cantero and Palmer \cite{MR3398727}, Berrick, Cohen, Wong and Wu \cite{MR2188127} introduced another way to add a new point by adding a point infinitesimally near an old point using the vector field. This allows Ellenberg and Wiltshire-Gordon \cite{jordan} to improve eventual polynomiality to immediate polynomiality of the betti numbers of PConf$_n(M)$ for some closed manifolds $M$. The following figures illustrate adding a point ``at infinity" and adding a point infinitesimally near an old point on the plane $\mathbb{R}^2$.

\begin{figure}[H]
\centering
  \includegraphics[scale=0.5]{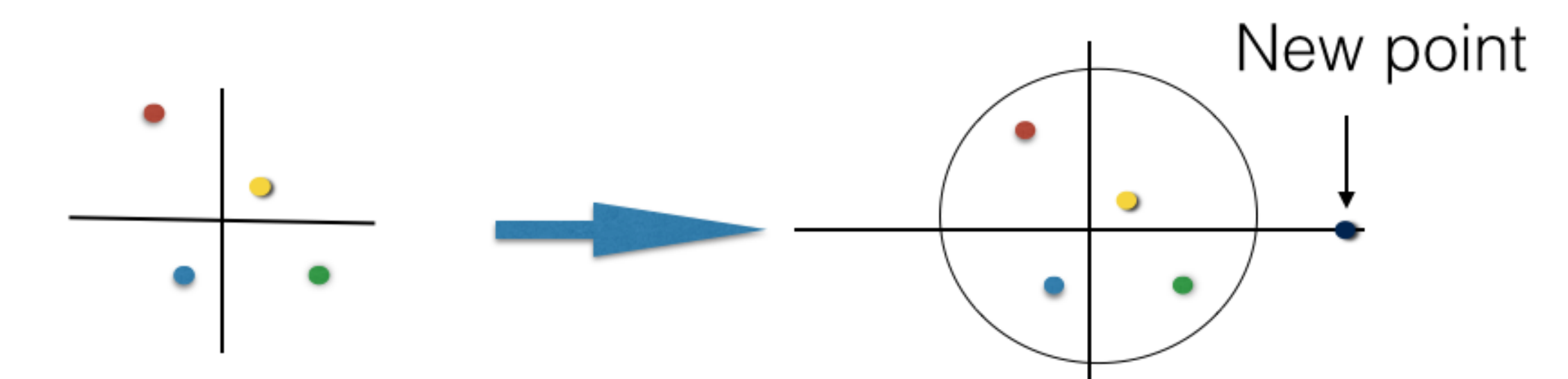}
  \caption{``adding a point at infinity"}
\label{neaki}
\end{figure}
\begin{figure}[H]
 \centering
  \includegraphics[scale=0.5]{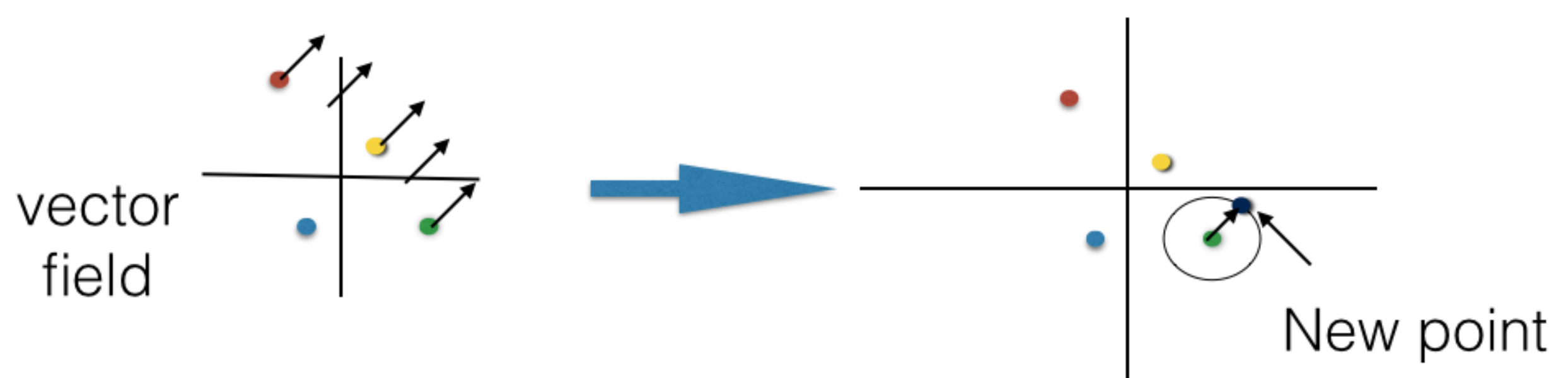}
  \caption{``adding a point near $x_k$"}
  \label{neark}
\end{figure}

 We call a section $s$ of $f_n(\mathbb{R}^2)$ (resp. $f_n(S^2)$) \emph{``adding a point near $x_k$"} if $s$ is homotopic to an element in the collection of sections Add$_{n,k}(\mathbb{R}^2)$ (resp. Add$_{n,k}(S^2)$). Informally, we assign $x_0$ at a sufficiently small distance to $x_k$ along some nonvanishing vector field. See Figure \ref{neark} for a demonstration of ``adding a point near $x_k$". Notice that there are infinitely many homotopy classes of sections in Add$_{n,k}(\mathbb{R}^2)$ and Add$_{n,k}(S^2)$ and they are classified by a kind of twists or sections of a circle bundle. See Section 2.1 and Section 3.2 for formal definitions of Add$_{n,k}(\mathbb{R}^2)$ and Add$_{n,k}(S^2)$ respectively.

We call a section $s$ of $f_n(\mathbb{R}^2)$ \emph{``adding a point at infinity"} if $s$ is homotopic to an element in the collection of sections Add$_{n,\infty}(\mathbb{R}^2)$; see Figure \ref{neaki}. Informally, we consider $\mathbb{R}^2$ as $S^2$ missing a point $\infty$, we can assign $x_0$ at a sufficiently small distance to $\infty$ along some nonvanishing vector field. See Section 2.1 for a formal definition of Add$_{n,\infty}(\mathbb{R}^2)$. 

Let $S_g$ be a surface of genus $g$ and $S^2$ be the $2$-sphere. In this paper, we will classify the sections of the fiber bundle $f_n(S)$ for 3 cases: $\mathbb{R}^2$, $S^2$ and $S_g$ when $g>1$. Here by \emph{section} we mean continuous section.
\begin{thm}[\bf\boldmath Classification of sections for ordered configurations]The following holds:

(1) If $S=\mathbb{R}^2$ and $n>3$, any section of $f_{n}(S)$ is either ``adding a point at infinity" or ``adding a point near $x_k$" for some $1\le k\le n$.

(2) If $S=S^2$ and $n=2$, the bundle $f_n(S)$ does not have a section. If $S=S^2$ and $n>4$, any section of $f_{n}(S)$ is ``adding a point near $x_k$" for some $1\le k\le n$.

(3) If $S=S_g$ a surface of genus $g>1$ and for $n>1$, the bundle $f_{n}(S)$ does not have a section.  
\label{main}
\end{thm}

For unordered case, we have the following corollary.

\begin{cor}[\bf\boldmath Classification of sections for unordered configurations]The following holds:

(1) If $S=\mathbb{R}^2$ and $n>3$, any section of $F_{n}(S)$ is ``adding a point at infinity";

(2) If $S=S^2$ and $n>4$ or $n=2,3$, the bundle $F_{n}(S)$ does not have a section.\label{main11}
\end{cor}
\begin{rmk}
We discuss the exceptional cases when $n=3$ for $S=S^2$ in Section \ref{exception}. Our method does not work for the case $n=4$ but \cite[Theorem 2]{MR2149513} proved that $F_4(S^2)$ does not have sections. The $g=1$ case seems to be more complicated to analyze, therefore we do not pursue here. Notice that the construction ``adding a point near $x_k$" works for the torus as well; see Section 2.1.
\end{rmk}

It is classical that $f_n(\mathbb{R}^2)$ admits a section. In \cite[Theorem 3.1]{FE}, Fadell showed that when $n>2$, the bundle $f_n(S^2)$ admits a section. The unordered case for $S=\mathbb{R}^2$, i.e. (1) of Corollary \ref{main11} has been proved by \cite[Main Theorem 2]{MR2253663} and \cite[Theorem 4]{MR3499033}. In \cite[Theorem 2]{MR2149513}, they prove the case (2) of Corollary \ref{main11}, and even stronger, they deal with the multi-section problems. All the previous proofs make use of the braid relation and the presentations of braid groups and do not imply (1) and (2) in Theorem \ref{main}. Our main novelty is to use the characterization of lantern relation in analyzing the canonical reduction systems. The canonical reduction system uses the Thurston classification of isotopy classes of diffeomorphisms of surfaces. This idea originated from \cite{MR726319}.
 
The ordered case for $S=S_g$ of $g>1$, i.e. (3) of Theorem \ref{main} has been proved by \cite[Theorem 2]{MR1977999}. Their proof makes heavily use of the presentations of surface braid group. We give a simpler proof using the cohomology of surface braid group and a classification theorem in \cite[Theorem 5]{lei1}.

\para{The structure of the paper}
\begin{itemize}
\item In Section 2, we introduce the construction and the main tool we use: canonical reduction system.
\item In Section 3, we reduce Theorem \ref{main}(1) to a more algebraic statement Theorem \ref{PB}.
\item In Section 4, we prove Theorem \ref{PB}, which is the main work of this paper.
\item In Section 5, we prove Theorem \ref{main}(2) by reducing it to Theorem \ref{main}(1).
\item In Section 6, we prove Theorem \ref{main}(3) by a classification of maps between configuration spaces of surfaces in \cite{lei1}.
\item In Section 7, we ask further questions.
\end{itemize}
\para{Acknowledgements}
The author would like to thank Kevin Casto for telling her about the construction of adding a nearby point and thank Nir Gadish for discussion. She would also like to thank Paolo Bellingeri, Dan Margalit, 
Cihan Bahran and an anonymous referee for pointing out many references of previous works on braid groups and typos. Finally, she would like to thank her advisor Benson Farb for his extensive comments and for his invaluable support from start to finish.

\section{The construction and background on canonical reduction systems}
Let $S$ be a surface and let PConf$_n(S)$ \emph{the pure configuration space} be the space of ordered n-tuple of distinct points on $S$. The natural embedding PConf$_n(S)\subset S^n$ gives the topology on PConf$_n(S)$. Let $f_n(S): \text{PConf}_{n+1}(S) \to \text{PConf}_n(S)$ be the map given by $f_n(x_0,x_1,\ldots ,x_n):=(x_1,\ldots ,x_n)$.

There is a natural action of permutation group $\Sigma_n$ on PConf$_n(S)$ by permuting the $n$ points. Thus the quotient space $\text{Conf}_n(S)$ is the space of unordered $n$-tuple of distinct points in $S$. Permutation group $\Sigma_n$ acts on the fiber bundle $f_n(S)$ as well. Let $F_n(S): \text{PConf}_{n+1}(S)/\Sigma_n \to  \text{PConf}_n(S)/\Sigma_n$ be the map given by $F_n(x_0,\{x_1,\ldots ,x_n\}):=\{x_1,\ldots ,x_n\}$. The subject of this section is to classify the sections of the fiber bundles $f_n(\mathbb{R}^2)$ and $F_n(\mathbb{R}^2)$.

\subsection{Constructing sections}
In this subsection we give constructions of sections of the fiber bundle $f_n(\mathbb{R}^2)$. There are two cases: ``adding a point near $x_k$" and ``adding a point at infinity". These constructions originate from Berrick, Cohen, Wong and Wu \cite{MR2188127}, but the idea appeared in \cite{MR0141126}.\\
\\
{\bf\boldmath Case 1: adding a point near $x_k$.} Define
\[\text{PConf}_{n,k}(\mathbb{R}^2)=\{(v_k,x_1,...,x_n)|x_1,...,x_n \text{ be } n \text{ points on } \mathbb{R}^2 \text{ and } v_k \text{ be a unit vector at } x_k\}.
\]
This is the total space of a circle bundle by forgetting the vector $ v_k$
\begin{equation}
S^1\to \text{PConf}_{n,k}(\mathbb{R}^2)\to \text{PConf}_{n}(\mathbb{R}^2).
\label{vector}
\end{equation}
Equip $\mathbb{R}^2$ with the Euclidean metric. Set
\[
\epsilon(x_1,...,x_n)=\frac{1}{2} \text{min}_{1\le i\neq j\le n}\{d(x_i,x_j)\}.
\]
By the definition of $\epsilon(x_1,...,x_n)$, setting $x_0$ to be the image of the $v_k$-flow at time $\epsilon(x_1,...,x_n)$ from $x_k$ gives a map:
\[
em_{n,k}(\mathbb{R}^2):\text{PConf}_{n,k}(\mathbb{R}^2)\hookrightarrow\text{PConf}_{n+1}(\mathbb{R}^2).
\] 
Composing a continuous section $s:\text{PConf}_{n}(\mathbb{R}^2)\to \text{PConf}_{n,k}(\mathbb{R}^2)$ of the fiber bundle (\ref{vector}) with $em_{n,k}(\mathbb{R}^2)$ gives a section of the fiber bundle $f_n(\mathbb{R}^2)$. 

\begin{defn}[\bf\boldmath Adding a point near $x_k$]
We denote by Add$_{n,k}(\mathbb{R}^2)$ the collection of sections of $f_n(\mathbb{R}^2)$ consisting of compositions of a section of \eqref{vector} with $em_{n,k}(\mathbb{R}^2)$.
\end{defn}
Notice that there are infinitely many homotopy classes of sections in Add$_{n,k}(\mathbb{R}^2)$ and they are in one-to-one correspondence with the homotopy classes of sections of \eqref{vector}.\\
\\
\noindent
{\bf\boldmath Case 2: adding a point at infinity.} Let us call the north pole of a 2-sphere the point at infinity $\infty$. Then $\mathbb{R}^2\cong S^2-\infty$ through the stereographic projection. Define
\[\text{PConf}_{n,\infty}(\mathbb{R}^2)=\{(v_\infty,x_1,...,x_n)|x_1,...,x_n \text{ be } n \text{ points on } \mathbb{R} \text{ and } v_\infty \text{ be a unit vector at } \infty\}.
\]

This is the total space of a circle bundle by forgetting the vector

\begin{equation}
S^1\to \text{PConf}_{n,\infty}(\mathbb{R}^2)\to \text{PConf}_{n}(\mathbb{R}^2).
\label{ivector}
\end{equation}

Equip $S^2$ with the spherical metric; i.e. the metric that is induced from the standard embedding $S^2\subset \mathbb{R}^3$. Set
\[
\epsilon(x_1,...,x_n)=\frac{1}{2} \text{min}_{1\le i\le n}\{d(x_i,\infty)\}.
\]
By the definition of $\epsilon(x_1,...,x_n)$, setting $x_0$ to be the image of the $v_{\infty}$-flow at time $\epsilon$ from $\infty$ gives a map:
\[
em_{n,\infty}(\mathbb{R}^2):\text{PConf}_{n,\infty}(\mathbb{R}^2)\hookrightarrow \text{PConf}_{n+1}(\mathbb{R}^2).
\] 
Composing a continuous section $s:\text{PConf}_{n}(\mathbb{R}^2)\to \text{PConf}_{n,\infty}(\mathbb{R}^2)$ of the fiber bundle (\ref{ivector}) with $em_{n,\infty}(\mathbb{R}^2)$ gives a section of the fiber bundle $f_n(\mathbb{R}^2)$.
\begin{defn}[\bf\boldmath Adding a point at infinity]
We denote by Add$_{n,\infty}(\mathbb{R}^2)$ the collection of sections of $f_n(\mathbb{R}^2)$ consisting of compositions of a section of \eqref{ivector} with $em_{n,\infty}(\mathbb{R}^2)$.
\end{defn}
Notice that there are infinitely many homotopy classes of sections in Add$_{n,\infty}(\mathbb{R}^2)$ and they are in one-to-one correspondence with the homotopy classes of sections of \eqref{ivector}.

\subsection{Background}
In this subsection we discuss some properties of canonical reduction systems and the lantern relation. Let $S=S_{g,p}^b$ be a surface with $b$ boundary components and $p$ punctures. Let Mod$(S)$ (reps. PMod$(S)$) be the \emph{mapping class group} (resp. \emph{pure mapping class group}) of $S$, i.e. the group of isotopy classes of orientation-preserving diffeomorphisms of $S$ fixing the boundary components pointwise and punctures as a set (resp. pointwise).  By ``simple closed curves", we often mean isotopy class of simple closed curves, e.g. by ``preserve a simple closed curve", we mean preserve the isotopy class of a curve. 

Thurston's classification of elements of Mod$(S)$ is a very powerful tool to study mapping class groups. We call a mapping class $f\in\text{Mod}(S)$ \emph{reducible} if a power of $f$ fixes a nonperipheral simple closed curve. Each nontrivial element $f\in\text{Mod}(S)$ is of exactly one of the following types: periodic, reducible, pseudo-Anosov. See \cite[Chapter 13]{BensonMargalit} and \cite{FLP} for more details. We now give the definition of canonical reduction system. 

\begin{defn}[{\bf Reduction systems}]
 A \emph{reduction system} of a reducible mapping class $h$ in $\text{\normalfont Mod}(S)$ is a set of disjoint nonperipheral curves that $h$ fixes as a set up to isotopy. A reduction system is \emph{maximal} if it is maximal with respect to inclusion of reduction systems for $h$. The \emph{canonical reduction system} $\text{\normalfont CRS}(h)$ is the intersection of all maximal reduction systems of $h$. 
\end{defn}
For a reducible element $f$, there exists $n$ such that $f^n$ fixes each element in CRS$(f)$ and after cutting out CRS$(f)$, the restriction of $f^n$ on each component is either periodic or pseudo-Anosov. See \cite[Corollary 13.3]{BensonMargalit}. Now we mention three properties of the canonical reduction systems that will be used later.
\begin{prop}
$\text{\normalfont CRS}(h^n)$=$\text{\normalfont CRS}(h)$ for any $n$.
\end{prop}
\begin{proof}This is classical; see  \cite[Chapter 13]{BensonMargalit}.
\end{proof}
For a curve $a$ on a surface $S$, denote by $T_a$ the Dehn twist about $a$. For two curves $a,b$ on a surface $S$, let $i(a,b)$ be the geometric intersection number of $a$ and $b$. For two sets of curves $P$ and $T$, we say that $S$ and $T$ \emph{intersect} if there exist $a\in P$ and $b\in T$ such that $i(a,b)\neq 0$. Notice that two sets of curves intersecting does not mean that they have a common element.
\begin{prop}
Let $h$ be a reducible mapping class in $\text{\normalfont Mod}(S)$. If $\{\gamma\}$ and $\text{\normalfont CRS}(h)$ intersect, then no power of $h$ fixes $\gamma$.
\label{noin}
\end{prop}

\begin{proof}
Suppose that $h^n$ fixes $\gamma$. Therefore $\gamma$ belongs to a maximal reduction system $M$. By definition, $\text{CRS}(h)\subset M$. However $\gamma$ intersects some curve in CRS$(f)$; this contradicts the fact that $M$ is a set of disjoint curves. 
\end{proof}
\begin{prop}
Suppose that $h,f\in \text{\normalfont Mod}(S)$ and $fh=hf$. Then $\text{\normalfont CRS}(h)$ and $\text{\normalfont CRS}(f)$ do not intersect.
\label{CRS(x,y)}
\end{prop}
\begin{proof}
By conjugation, we have that CRS($hfh^{-1})=h(\text{CRS}(f))$. Since $hfh^{-1}=f$, we get that CRS$(f)=h(\text{CRS}(f))$. Therefore $h$ fixes the whole set CRS$(f)$. A power of $h$ fixes all curves in CRS$(f)$. By Proposition \ref{noin}, curves in CRS$(h)$ do not intersect curves in CRS$(f)$.
\end{proof}

Now, we introduce a remarkable relation for $\text{\normalfont Mod}(S)$ that will be used in the proof.
\begin{prop}[\bf\boldmath The lantern relation]
There is an orientation-preserving embedding of $S_{0,4}\subset S$ and let $x,y,z,b_1,b_2,b_3,b_4$ be simple closed curves in $S_{0,4}$ that are arranged as the curves shown in the following figure. 

\begin{figure}[H]
\centering
\begin{tikzpicture}
\draw[\discColor, thick] (0,0) circle [radius = 1*\discSizeX];
\draw[green, thick] (0,0) circle [radius = 0.9*\discSizeX];
\node at (50:1.15*\discSizeX) {$b_4$};
\draw[green, thick] (90:0.53*\discSizeX) circle [radius = 0.15*\discSizeX];
\draw[\discColor, thick] (90:0.53*\discSizeX) circle [radius = 0.1*\discSizeX];
\node at (90:0.25*\discSizeX) {$b_1$};
\draw[green, thick] (210:0.53*\discSizeX) circle [radius = 0.15*\discSizeX];
\draw[black, thick] (210:0.53*\discSizeX) circle [radius = 0.1*\discSizeX];
\node at (210:0.25*\discSizeX) {$b_2$};
\draw[green, thick] (-30:0.53*\discSizeX) circle [radius = 0.15*\discSizeX];
\draw[black, thick] (-30:0.53*\discSizeX) circle [radius = 0.1*\discSizeX];
\node at (-30:0.25*\discSizeX) {$b_3$};
\draw[\curveColor, ultra thick] (150:0.25*\discSizeX) circle [x radius = 0.75*\discSizeX, y radius = 0.3*\discSizeX, , rotate=60];
\node[color=\curveColor] at (150:0.7*\discSizeX) {$x$};
\draw[\curveColor, ultra thick] (30:0.25*\discSizeX) circle [x radius = 0.75*\discSizeX, y radius = 0.3*\discSizeX, , rotate=-60];
\node[color=\curveColor] at (30:0.7*\discSizeX) {$z$};
\draw[\curveColor, ultra thick] (-90:0.25*\discSizeX) circle [x radius = 0.75*\discSizeX, y radius = 0.3*\discSizeX];
\node[color=\curveColor] at (-90:0.7*\discSizeX) {$y$};
\end{tikzpicture}
\end{figure}
In {\normalfont Mod}$(S)$ we have the relation
\[T_xT_yT_z=T_{b_1}T_{b_2}T_{b_3}T_{b_4}.\]
\end{prop}
\begin{proof}
This is classical; see \cite[Chapter 5.1]{BensonMargalit}.
\end{proof}

\section{An algebraic result and how it implies (1) of Theorem \ref{main}}
In this section we give an algebraic result about the braid groups and prove how it implies (1) of Theorem \ref{main}. PConf$_n(\mathbb{R}^2)$ and PConf$_{n+1}(\mathbb{R}^2)$ are both $K(\pi,1)$ spaces. This can be seen by induction on $n$ and taking the long exact sequence of homotopy groups of the fiber bundle $f_n(\mathbb{R}^2)$. Therefore, the homotopy classes of sections of $f_n(\mathbb{R}^2)$ only depend on the homomorphisms of the fundamental groups. Let $PB_n=\pi_1(\text{PConf}_n(\mathbb{R}^2))$ and let $F_n$ be a free group of $n$ generators. The fundamental groups of the fiber bundle $f_n(S)$ gives us the following short exact sequence, i.e. the Fadell-Neuwirth short exact sequence:

\begin{equation}
1\to F_n\to PB_{n+1}\xrightarrow{f_{n}(\mathbb{R}^2)_*} PB_n\to 1.
\label{exact}
\end{equation}

Let $D_n$ be the disk with $n$ punctures $\{x_1,...,x_n\}$ and $D_{n+1}$ be the disk with $n+1$ punctures $\{x_0, x_1,...,x_n\}$  and the forget map forgets the point $x_0$. We view $PB_n$ and $PB_{n+1}$ as mapping class groups as the following:
\[
PB_n=\text{PMod}(D_n)\text{ and }PB_{n+1}=\text{PMod}(D_{n+1}).
\]
A simple closed curve $a$ on $D_n$ separates $D_n$ into two parts: the \emph{outside of $a$}, i.e. the component containing the boundary of $D_n$ and the \emph{inside of $a$}, i.e. the one not containing the boundary of $D_n$. We say that $a$ \emph{surrounds} $x_k$ if $x_k\in$ the inside of $a$. The following algebraic result on the splittings of the exact sequence (\ref{exact}) is a key ingredient in the proof of Theorem \ref{main}. 

\begin{thm}
Suppose that we have a section $s: PB_n\to PB_{n+1}$. Then the image $s(PB_n)$ either preserves a simple closed curve $c$ surrounding points $\{x_1,...,x_n\}$, or preserves a simple closed curve $c$ surrounding $\{x_i,x_0\}$ for some $i\in \{1,2,...,n\}$. 
\label{PB}
\end{thm}

The rest of this subsection focuses on how Theorem \ref{PB} implies part (1) of Theorem \ref{main}. Let $c$ be a curve inside $D_{n+1}$ surrounding $k$ points. Let $D_k^l$ be a disk with $k$ punctures and $l$ open disks removed. We call the boundary of the $l$ disks the \emph{small boundary components} and the original boundary of $D$ \emph{the big boundary component}. See the following figure for a geometric explanation.
\begin{figure}[H]
\centering
\includegraphics[scale=0.25]{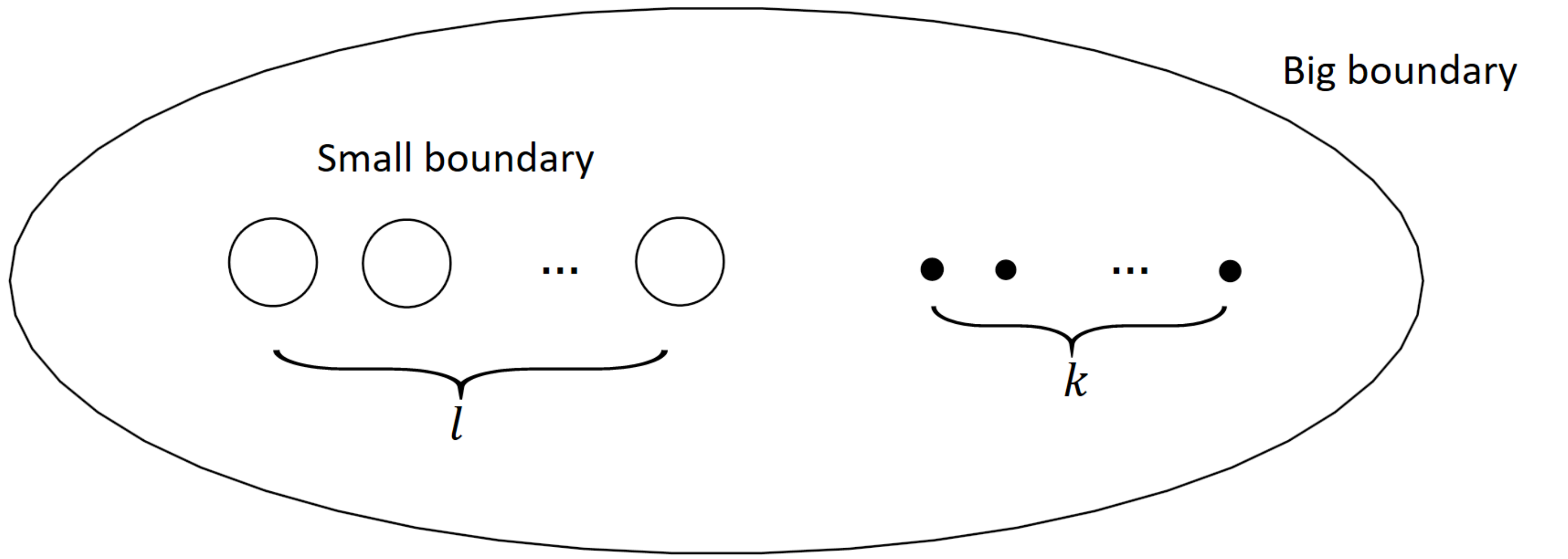}
\caption{$D_k^l$ where small boundaries are the $l$ small circles and big boundary is the outside circle.}
\end{figure}
Let 
\[
PB_{k,l}:=\text{PMod}(D_k^l)
\]
 be the pure mapping class group of $D_k^l$. The difference between punctures and boundary components is that the Dehn twist about a puncture is trivial but the Dehn twist about a boundary component is nontrivial. The following proposition describes the centralizer of $T_c$. Denote the centralizer of $T_c$ by $C_{PB_{n+1}}(T_c)$.  
\begin{prop}[\bf\boldmath Centralizer of $T_c$]
\label{center}
$C_{PB_{n+1}}(T_c)$ satisfies the following exact sequence
\[
1\to \mathbb{Z}\xrightarrow{(T_c,T_c^{-1})} PB_k\times PB_{n+1-k,1} \to C_{PB_{n+1}}(T_c)\to 1\]
where $k$ is the number of points that $c$ surround.
\end{prop}
\begin{proof}
This is classical. The centralizer of $T_c$ is the subgroup of Mod$(D_n)$ that fixes $c$. The curve $c$ separates $D_n$ into two components: $C_1$ that contains the boundary and $C_2$ that does not contain the boundary. Since $C_1$ and $C_2$ are not homeomorphic, we have that $C_{PB_{n+1}}(T_c)$ only contains elements that preserve $C_1$ and $C_2$. Therefore, our statement holds.
\end{proof}

Now we are ready to prove (1) of Theorem \ref{main}.

\begin{proof}[\bf\boldmath Proof of (1) of Theorem \ref{main} assuming Theorem \ref{PB}] 
Let $g:\text{PConf}_n(\mathbb{R}^2)\to \text{PConf}_{n+1}(\mathbb{R}^2)$ be a section of the fiber bundle $f_n(\mathbb{R}^2)$. Let $s=g_*: PB_n\to PB_{n+1}$ be the induced map on the fundamental groups of $g$. By Theorem \ref{PB}, the image $s(PB_n)$ preserves a curve $c$ that either surrounds 2 points or $n$ points. Therefore, $s(PB_n)$ is in the centralizer of $T_c$ in $PB_{n+1}$ by the fact that $fT_cf^{-1}=T_{f(c)}$.\\
\\
\noindent
{\bf\boldmath Case 1: when $c$ surrounds $\{x_0,x_k\}$.} $PB_2\cong\mathbb{Z}$, which is generated by the Dehn twist about the boundary component. From Proposition \ref{center} we have
\[
1\to \mathbb{Z}\xrightarrow{(T_c,T_c^{-1})} \mathbb{Z}\times PB_{n-1,1}\to C_{PB_{n+1}}(T_c)\to 1.
\]
Therefore $C_{PB_{n+1}}(T_c)\cong PB_{n-1,1}$. The inclusion $PB_{n-1,1}\hookrightarrow PB_{n+1}$ is induced by gluing a 2-punctured disk inside the small boundary of $D_{n-1}^1$.

On the other hand, we have that
\[
\pi_1(\text{PConf}_{n,k}(\mathbb{R}^2))=PB_{n-1,1}.
\]
The fundamental groups of the fiber bundle (\ref{vector}) is the following exact sequence:
\begin{equation}
1\to \mathbb{Z}\xrightarrow{T_d} PB_{n-1,1}\to PB_n\to 1.
\label{exact2}
\end{equation}
Here $T_d$ is the Dehn twist about the small boundary component. The embedding $em_{n,k}$: PConf$_{n,k}(\mathbb{R}^2)\hookrightarrow \text{PConf}_{n+1}(\mathbb{R}^2)$ induces a homomorphism on the fundamental group $em_{n,k*}:PB_{n-1,1}\to PB_{n+1}$. On the mapping class group level, since $T_d$ in $PB_{n-1,1}$ is mapped to the Dehn twist about a curve surrounding $\{x_0,x_k\}$, we know that $em_{n,k*}$ is also induced by gluing a 2-punctured disk inside the small boundary of $D_{n-1}^1$. The theorem holds.\\
\\
{\bf\boldmath Case 2: when $c$ surrounds $\{x_1,...,x_n\}$.} Since $PB_{1,1}\cong\mathbb{Z}\times \mathbb{Z}$, which is generated by the Dehn twists about the two boundaries, we have the following exact sequence:
\[
1\to \mathbb{Z}\xrightarrow{(0,T_c,T_c^{-1})}\mathbb{Z}\times \mathbb{Z}\times PB_n\to C_{PB_{n+1}}(T_c)\to 1.
\]
On the mapping class group level, $PB_n\times PB_{1,1}\to C_{PB_{n+1}}(T_c)\to PB_{n+1}$ is induced by gluing $D_1^1$ outside the big boundary component of $D_n$. Therefore $\mathbb{Z}\times PB_n\cong C_{PB_{n+1}}(T_c)$ and the generator of $\mathbb{Z}$ is mapped to $T_cT_b^{-1}$ where $b$ is the big boundary of $D_{n+1}$.

On the other hand, we have that
\[
\pi_1(\text{PConf}_{n,\infty}(\mathbb{R}^2))=\mathbb{Z}\times PB_n.
\]
The embedding $em_{n,\infty}:$ PConf$_{n,\infty}(\mathbb{R}^2)\hookrightarrow \text{PConf}_{n+1}(\mathbb{R}^2)$ induces $em_{n, \infty*}:\mathbb{Z}\times PB_n\to PB_{n+1}$ on the fundamental groups. On the level of mapping class groups, since $\mathbb{Z}$ maps to $T_cT_b^{-1}$, we know that $em_{n,\infty*}$ is induced by the embedding of $D_n$ in $D_{n+1}$ and maps the generator of $\mathbb{Z}$ to $T_cT_b^{-1}$. Therefore, $em_{n,\infty*}$ is induced by gluing $D_1^1$ outside the big boundary component of $D_n$ as well. Our theorem holds.
\end{proof}
\begin{rmk}
The classification of the sections of the fiber bundle $f_n(S)$ is not entirely the same as the classification of the splittings of the exact sequence (\ref{exact}). There is an subtlety coming from the choice of base point in the fundamental groups. Therefore, we classify the splittings of the exact sequence (\ref{exact}) up to conjugacy. In Theorem \ref{PB}, all the choices of $c$ is coming from a conjugacy by an element $F_n$; thus they decide the same sections. 
\end{rmk}

\section{The proof of Theorem \ref{PB}}
Throughout the section we prove Theorem \ref{PB}, which implies Theorem \ref{main}(1). The strategy of the proof is the following. We assume that there exists a section $s:PB_n\to PB_{n+1}$, i.e. $f_{n}(\mathbb{R}^2)_*\circ s=id$. The strategy is that we first determine $s(T_a)$ for any simple closed curve $a$ on $D_n$. We first prove that the lift $s(T_a)$ is always a multi-twist about at most two curves on $D_{n+1}$; these two curves or one curve are either trivial or isotopic to $a$ after forgetting the point $x_0$. This is done by using a result of McCarthy on centralizer of pseudo-Anosov element and lantern relation. We find a generating set of $PB_n$ consisting of Dehn twists about curves bounding two points. We then argue depending on whether $s(T_a)$ is a multi-twist on two curves or a single twist. The main tool of this part is Proposition \ref{2int}, characterizing lantern relation that we deduce from Thurston's construction.

\subsection{Step 1: constrain the image of $s(T_c)$ for a simple closed curve $c$}
 The following proposition characterizes intersection number 2 of two curves and will be used many times in the proof.
\begin{prop}
Let $i(a,b)\neq 0$. Then $T_aT_b$ is a multitwist if and only if $i(a,b)=2$.
\label{2int}
\end{prop}
\begin{proof}
This result was previously obtained by Margalit \cite{MR1943337} and Hamidi-Tehrani \cite{MR1943336}. We give a different proof using Thurston's construction; see e.g. \cite[Theorem 14.1]{BensonMargalit}. There is a subspace $T$ of $S$ that $a,b$ fills, i.e. the tubular neighborhood of $a\cup b$. Let $\langle T_a,T_b\rangle$ be the group generated by $T_a$ and $T_b$ in Mod$(T)$. Thurston's theorem says that when $a,b$ fill, there is a representation $\rho: \langle T_a,T_b\rangle\to \text{PSL}(2,\mathbb{R})$ such that 
\[
T_a\to \begin{bmatrix}
1 & -i(a,b)\\
0 & 1
\end{bmatrix}
\text{ and }
 T_b\to \begin{bmatrix}
1 & 0\\
i(a,b) & 1
\end{bmatrix}.
\]
$\rho(h)$ is parabolic if and only if $h$ is reducible on $T$. We know that 
\[
\rho(T_aT_b)= \begin{bmatrix}
1 & -i(a,b)\\
0 & 1
\end{bmatrix}\begin{bmatrix}
1 & 0\\
i(a,b) & 1
\end{bmatrix}=\begin{bmatrix}
1-i(a,b)^2 & -i(a,b)\\
i(a,b) & 1
\end{bmatrix}
\]
Since Trace$(\rho(T_aT_b))=2-i(a,b)^2$, we know that $T_aT_b$ is reducible on $T$ if and only if $i(a,b)=2$. By the lantern relation, we know that $T_aT_b$ is a multitwist when $i(a,b)=2$.
\end{proof}

The following lemma determines $s(T_a)$ for any simple closed curve $a$ on $D_n$.

\begin{lem}[\bf\boldmath The lift of a Dehn twist]
Let $a$ be a simple closed curve on $D_n$, then $s(T_a)$ can only be one of the following three cases:

(1) It can be a Dehn twist $T_{a'}$ about a curve $a'$ on $D_{n+1}$ such that after forgetting $x_0$, we have $a'=a$.

(2) It can be a multitwist $T_{a'}T_c^m$ (i.e. a product of twists on disjoint curves) about two curves $a'$ and $c$ on $D_{n+1}$ for $m\in \mathbb{Z}$, where $c$ surrounds $2$ points $\{x_0,x_k\}$ and after forgetting $x_0$, we have that $a'=a$.

(3) It can be $T_{a'}(T_{a'}T_{a''}^{-1})^n$, where $a'$ and $a''$ are disjoint on $D_{n+1}$  such that after forgetting $x_0$, we have that $a'=a''=a$.
\label{lift}
\end{lem}
\begin{proof}
We start with a the proof of the claim:  After forgetting $x_0$, any element of CRS$(s(T_a))$ is either $a$ or surrounding one puncture (trivial). 

The centralizer of $T_a$ contains $T_b$ when $a,b$ are disjoint curves. By injectivity of $s$, the centralizer of $s(T_a)$ contains a copy of $\mathbb{Z}^2$ when $n>3$. By \cite[Theorem 1]{McCarthy} that the centralizer of a pseudo-Anosov element is virtually cyclic, we know that $s(T_a)$ is not pseudo-Anosov. The injectivity of $s$ also implies that $s(T_a)$ is not a torsion element. Therefore we have that $s(T_a)$ is reducible under Thurston's classification of mapping classes. Assume there exists $b'\in \text{CRS}(s(T_a))$ such that after forgetting $x_0$, we have that $b$ is not trivial and $b\neq a$. Since $b'\in \text{CRS}(s(T_a))$, we have that a power of $s(T_a)$ fixes $b'$. Also a power of any mapping class that commutes with $s(T_a)$ fixes $b$ as well. We break our discussion into the following two cases.
 \begin{itemize}
 \item
{\bf Case 1:} If $i(a,b)\neq 0$, then no power of $T_a$ fixes $b$. However we also have some power of $s(T_a)$ fixes $b'$. This is a contradiction.
\item
{\bf Case 2:} If $i(a,b)=0$ but $b\neq a$, then there exists a curve $c$ such that $i(c,b)\neq 0$ but $i(c,a)=0$. Since $s(T_c)$ commutes with $s(T_a)$, we know that $s(T_c)$ preserves CRS$(s(T_a))$. However $i(b,c)\neq 0$, which shows that no power of $T_c$ preserve $b$. This contradicts the fact that a power of  $s(T_c)$ preserves $b'$.
\end{itemize}

By the disjointness of curves in canonical reduction system, we have that CRS$(s(T_a))$ contains at most 2 curves. We break the rest of the proof into 3 cases.
 \begin{itemize}
 \item
{\bf\boldmath Case 1: CRS$(s(T_a))$ only contains one curve $a'$.} It depends on the location of $x_{0}$, only one side of $a'$ will contain $x_{0}$, which means only that side could $s(T_a)$ could possibly be not identity. If $a'$ surrounds $x_0$ and $a$ surrounds more than $2$ points, there is a curve $b$ inside of $a$ containing 2 points, therefore $s(T_a)$ fixes CRS$(s(T_b))$ inside of $a'$, which means $s(T_a)$ does not acts as pseudo-Anosov inside $a'$. This proves that a power of $s(T_a)$ is the identity on the inside. Therefore a power of $s(T_a)$ is a power of the $T_{a'}$. Since $s(T_a)$ is a lift of $T_a$, we have that a power of $s(T_a)T_{a'}^{-1}$ is the identity. Therefore $s(T_a)=T_{a'}$.

If $a'$ surrounds $x_0$ and $a$ surrounds $2$ points, we position $a$ as in the Figure \ref{LEMMA2.12(1)}.
\begin{figure}[H]
\centering
  \includegraphics[scale=0.35]{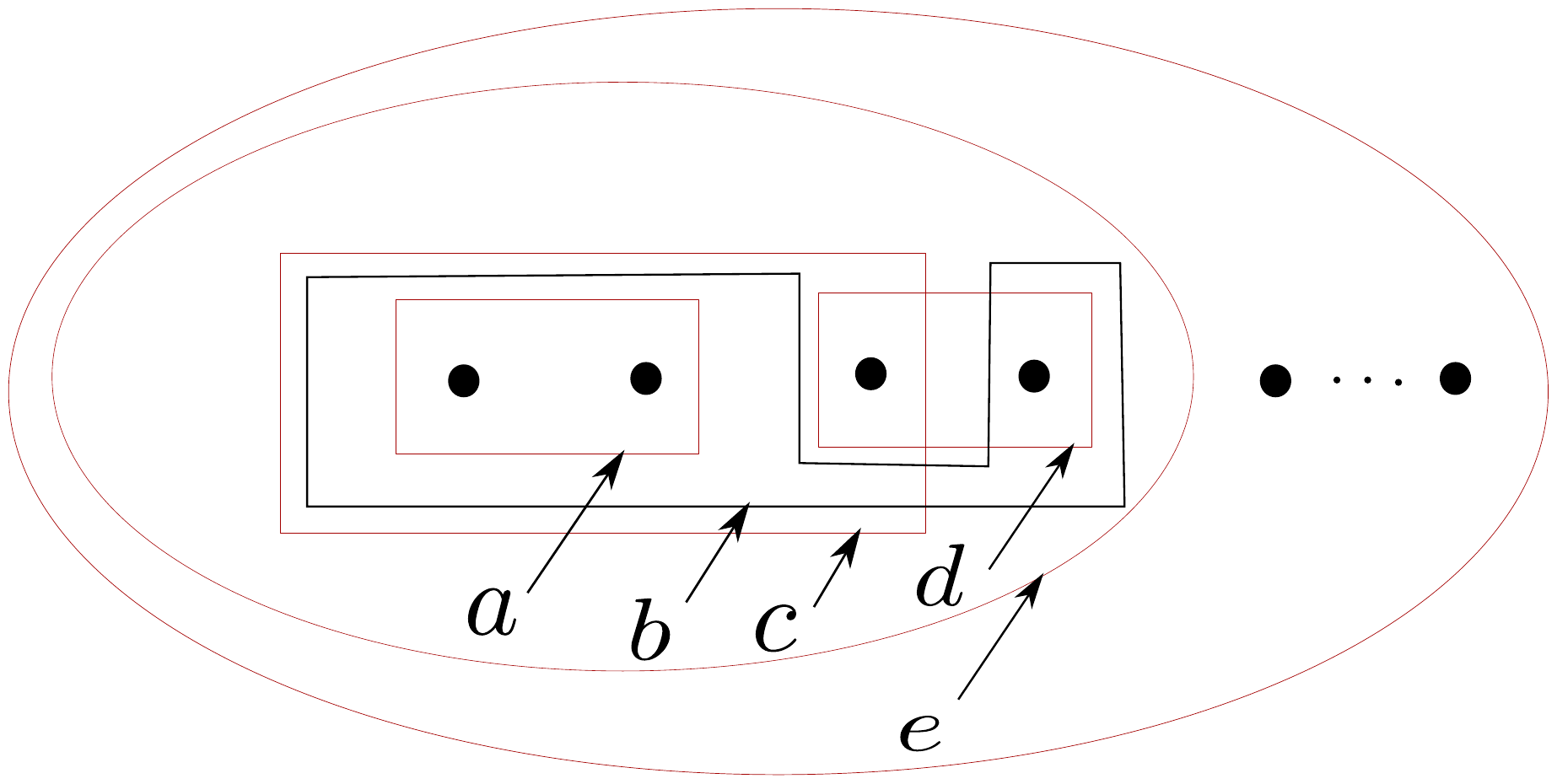}
  \caption{$T_cT_dT_b=T_aT_e$}
  \label{LEMMA2.12(1)}
\end{figure}
By the lantern relation, we have
\begin{equation}
T_cT_dT_b=T_aT_e.
\label{lr}
\end{equation}
Since $T_b,T_c,T_d,T_e$ commutes with $T_a$, we have that $s(T_b),s(T_c),s(T_d),s(T_e)$ fix $a'$ and therefore is identity on the component that $a'$ stay. Since $s(T_a)$ is a product of $s(T_b),s(T_c),s(T_d),s(T_e)$, we know that $s(T_a)$ is also identity in the interior of $a'$. Therefore $s(T_a)=T_{a'}$. The case when $a'$ does not surround $x_0$ is similar.

\item
{\bf\boldmath Case 2:  CRS$(s(T_a))$ only contains two curves $a'$ and $c$ such that $c$ surrounds $2$ points $\{x_0,x_k\}$. } On both the exterior of $a'$ and the interior of $a'$ pinching curve $c$, we know that $s(T_a)$ is identity. Therefore we know that $s(T_a)$ is the multi-twist on $a',c$. 
\item
{\bf\boldmath Case 3:  CRS$(s(T_a))$ only contains two curves $a'$ and $a''$ such that after forgetting $x_0$, both curves $a'$ and $a''$ become $a$. } On both the interior of $a'$ and the exterior of $a''$, we know that $s(T_a)$ is identity. Therefore we know that $s(T_a)$ is the multi-twist on $a',a''$. \qedhere
\end{itemize}
\end{proof}

\begin{nota}
In the following argument, we will use small letters like $a,b,c,...$ to represent simple closed curves on $D_n$ and small letters with a prime or double primes like $a',a'',b', ...$ to represent the canonical reduction systems of $s(T_a), s(T_b),...$. If we have two curves in $\text{\normalfont CRS}(s(T_a))$, we use $a'$ and $a''$. 
\end{nota}
\subsection{Step 2: the case of adding points at infinity}
On $D_n$, we call a simple closed curve surrounding 2 points by \emph{a basic simple closed curve}. The following lemma gives one condition for $s(PB_n)$ to preserve a simple closed curve surrounding $\{x_1,...,x_n\}$.

\begin{lem}If the canonical reduction system of any basic simple closed curve does not contain a curve surrounding $x_0$, then $s(PB_n)$ preserves a simple closed curve surrounding $\{x_1,...,x_n\}$.
\label{basic}
 \end{lem}
 
\begin{proof}
Suppose that there exists a simple closed curve $a$ such that $\text{\normalfont CRS}(a)$ contains a curve surrounding $x_0$. We call $a$ the \emph{innermost} if $a$ surrounds $k$ points and the canonical reduction systems of all curves surrounding $k-1$ points does not contains a curve surrounding $x_0$. Take an innermost curve $a$ such that $a$ surrounds $k$ points in $D_{n}$. By the assumption of Lemma \ref{basic}, we have that $k>2$. There are three cases according to Lemma \ref{lift}.
\begin{itemize}
\item
{\bf\boldmath Case 1: $\text{\normalfont CRS}(a)=\{a'\}$ such that after forgetting $x_0$, we have $a'=a$.} We take $b$ and $c$ inside $D_n$ as in the following figure, we have the lantern relation $T_bT_c=T_eT_aT_d^{-1}$. 
\begin{figure}[H]
\centering
 \includegraphics[scale=0.3]{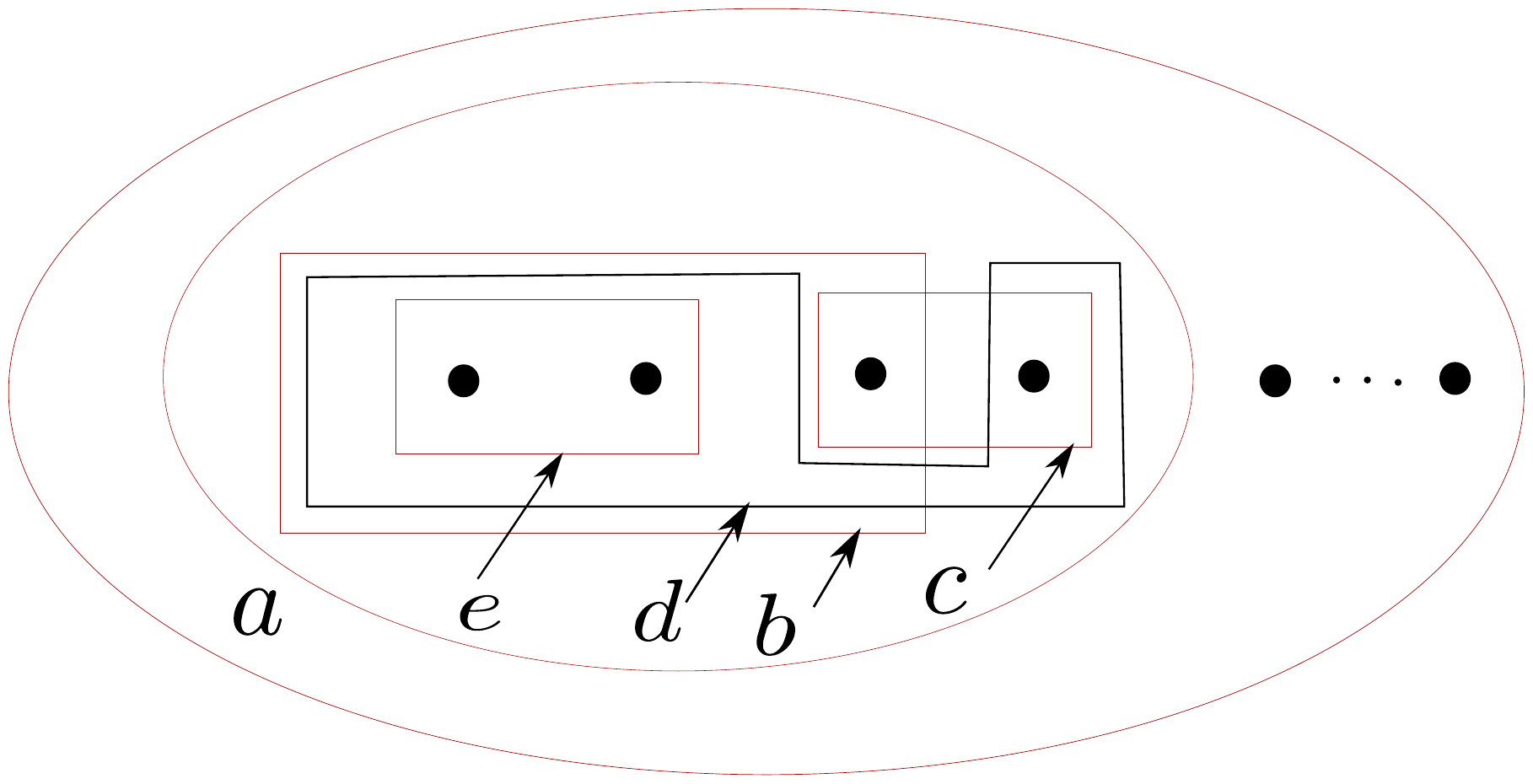}
   \caption{$T_bT_c=T_eT_aT_d^{-1}$}
  \label{figure10}
\end{figure}

Because $b,c,d,e$ surround less points than $k$ and $a$ is the innermost curve, we know that $\text{\normalfont CRS}(s(T_b))$, $\text{\normalfont CRS}(s(T_c))$, $\text{\normalfont CRS}(s(T_d))$ and $\text{\normalfont CRS}(s(T_e))$ each only contains one curve not surrounding $x_0$ and we denote them by $b',c',d',e'$. Since $T_e$, $T_a$ and $T_d$ commute with each other, their canonical reduction systems would be disjoint. By Lemma \ref{lift}, we know that $s(T_eT_aT_d^{-1})$ is also a multitwist. Therefore by Lemma \ref{2int}, we know that $i(b',c')=2$. However $\text{\normalfont CRS}(T_{b'}T_{c'})$ does not contain $a'$ because $a'$ surround $x_0$ but $b',c'$ do not. This is a contradiction.
\item
{\bf\boldmath Case 2: CRS$(a)=\{a',a''\}$ such that $a''$ surrounds $2$ points $\{x_0,x_k\}$ and after forgetting $x_0$, we have that $a'=a$.}\\
\item
{\bf\boldmath Case 3: CRS$(a)=\{a',a''\}$ such that after forgetting $x_0$, we have that $a'=a''=a$.}\\
\\
Case 2 and 3 can be proved using a similar argument as in Case 1. We construct the same lantern relation and use the fact that $a$ is the innermost curve to reach a contradiction. Therefore if the canonical reduction systems of all basic simple closed curves do not contain a curve surrounding $x_0$, then the canonical reduction systems of any curve does not surround $x_0$. This is true for the center element of $PB_n$ as well. Let $c$ be the boundary curve of $D_n$. Then CRS$(c)=\{c'\}$ does not contain $x_0$. However, all Dehn twists commute with $T_c$ which preserves $c'$. \qedhere
\end{itemize}
\end{proof}

Now we introduce a generating set for $PB_n$. Consider the $n$-punctured disk $D_n$ in Figure \ref{51}. Let $L$ be a segment below all the other points. Let $L_1,...,L_n$ be segments connecting $x_1,...,x_n$ to the segment $L$. Figure \ref{52} is the corresponding figure for $D_{n+1}$.
\begin{figure}[H]

\minipage{0.33\textwidth}\centering
  \includegraphics[scale=0.17]{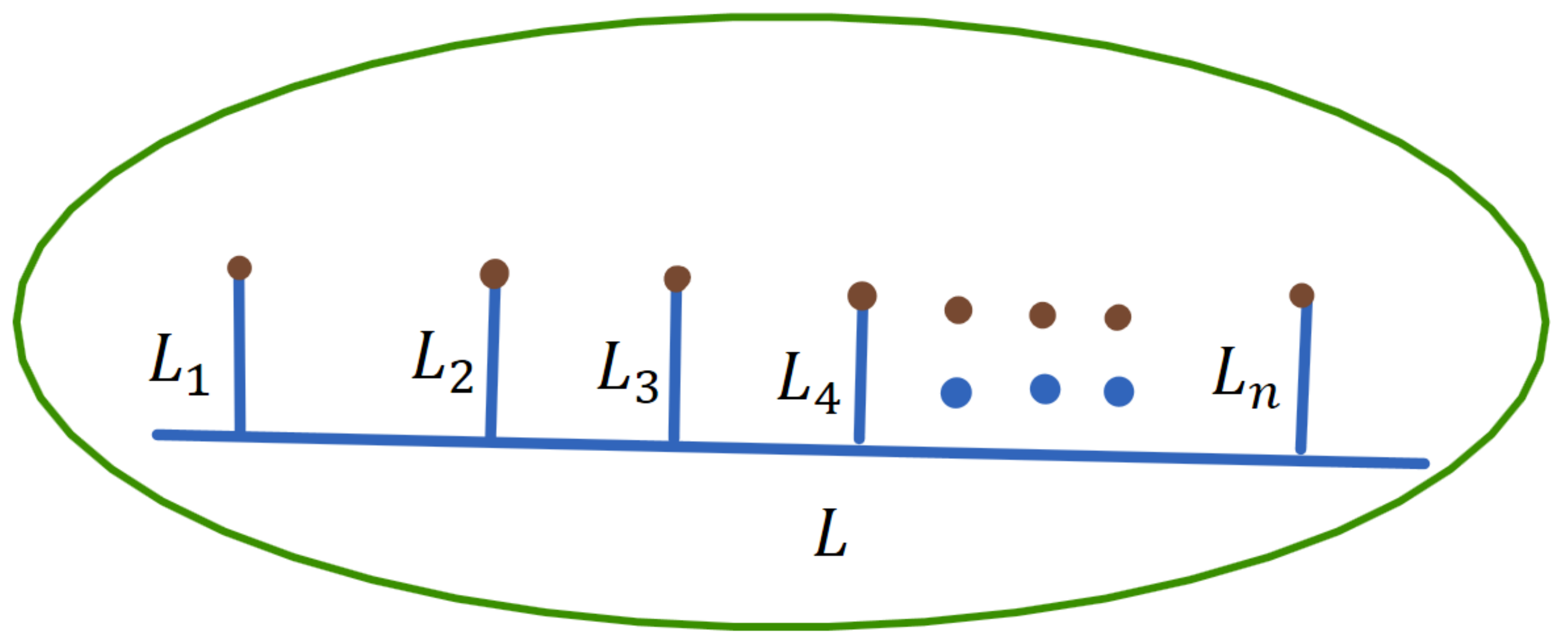}\label{51}
  \caption{$D_{n}$.}\label{51}
\endminipage\hfill
\minipage{0.33\textwidth}\centering
  \includegraphics[scale=0.17]{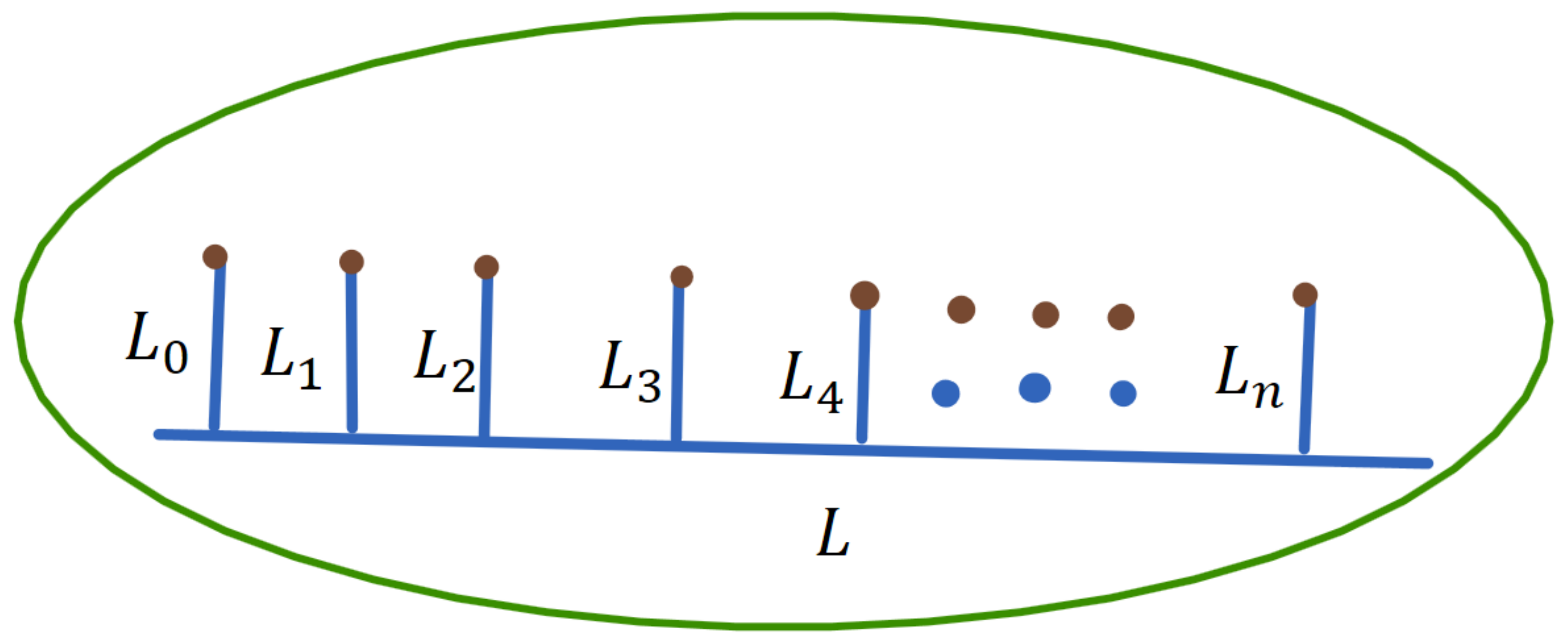}
  \caption{$D_{n+1}$.}\label{52}
\endminipage\hfill
\minipage{0.33\textwidth}\centering
  \includegraphics[scale=0.17]{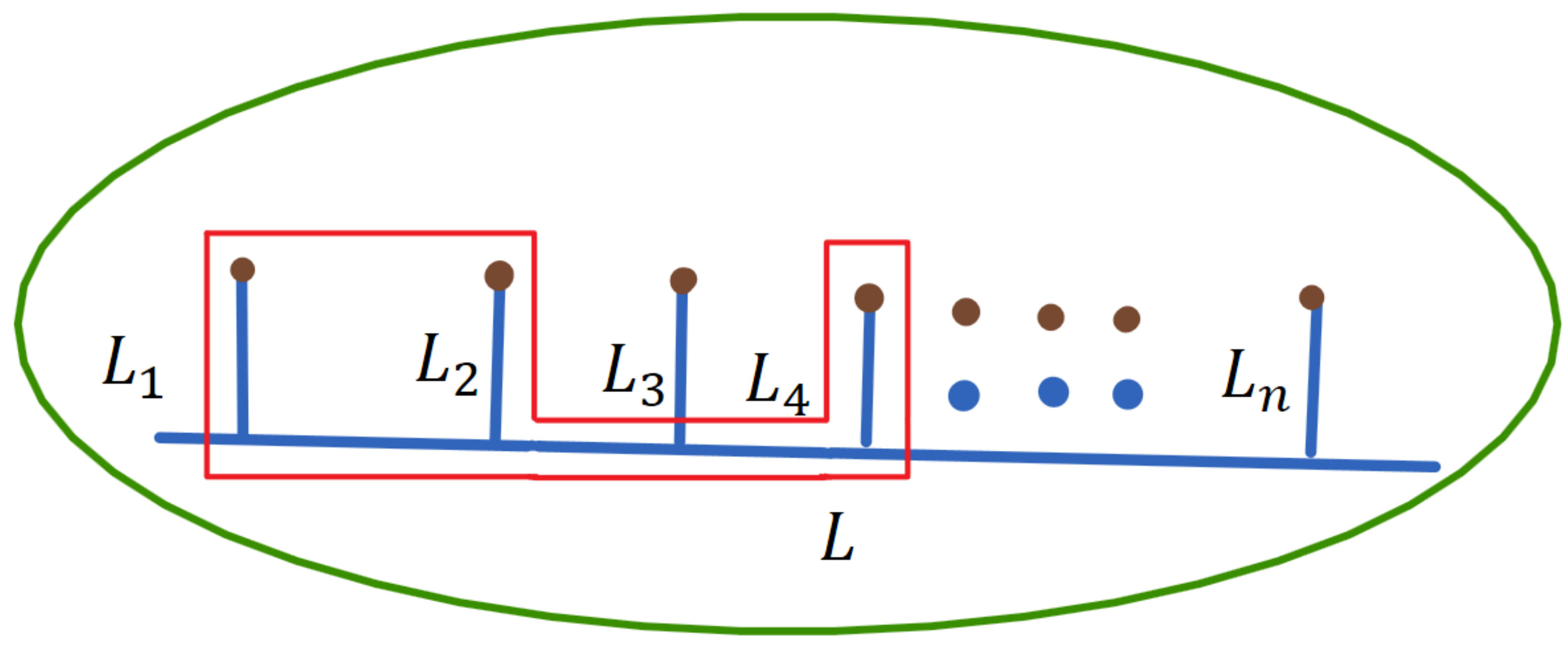}
  \caption{An example of Notation \ref{notation} for $a_{124}$.}
  \label{1000}
\endminipage\hfill
\end{figure}
\begin{nota}
For $\{i_1,...,i_k\}$ a subset of $\{1,...,n\}$, let $a_{i_1i_2...i_k}$ be the boundary curve of the tubular neighborhood of $L\cup \cup_{m=1}^{k}L_{i_k}$. Denote by $A_{i_1i_2...i_k}$ the Dehn twist about $a_{i_1i_2...i_k}$.  For $\{i_1,...,i_k\}$ a subset of $\{0,1,...,n\}$, let $b_{i_1i_2...i_k}$ be the boundary curve of the tubular neighborhood of $L\cup \cup_{m=1}^{k}L_{i_k}$. Denote by $B_{i_1i_2...i_k}$ the Dehn twist about  $b_{i_1i_2...i_k}$. See Figure \ref{1000} for an example of a curve representing $a_{124}$.
\label{notation}
\end{nota}

The following proposition describes a generating set of the group $PB_n$.
 \begin{prop}
There is a generating set of $PB_n$ consisting of all the Dehn twists about the basic curves $a_{ij}$ for $1\le i<j\le n$.
\end{prop}
\begin{proof}
This is classical and can be prove it by induction on the exact sequence
\[1\to F_k\to PB_{k+1}\to PB_k\to 1.\]
This generating set is given by Artin; e.g. see \cite[Theorem 2.3]{MR2490001}
\end{proof}

\subsection{Step 3: Finishing the proof of Theorem \ref{PB}} In this subsection, we prove Theorem \ref{PB}. We break the proof into several cases. By Lemma \ref{basic}, we only need to consider the case that there exists at least one basic simple closed curve $a$ such that some element of CRS$(a)$ surrounds $x_0$. We break our discussion into the following four cases.

\para{Case 1: The canonical reduction systems of any basic curves only contain one curve $a'$ and $a'$ surrounds $x_0$.}  
\begin{proof}[Proof of Case 1]
Let $a,b,c,d$ be the curves in Figure \ref{CASE1}. We have the lantern relation $T_aT_bT_c=T_d$. 
\begin{figure}[H]
\minipage{0.48\textwidth}
\centering
  \includegraphics[scale=0.4]{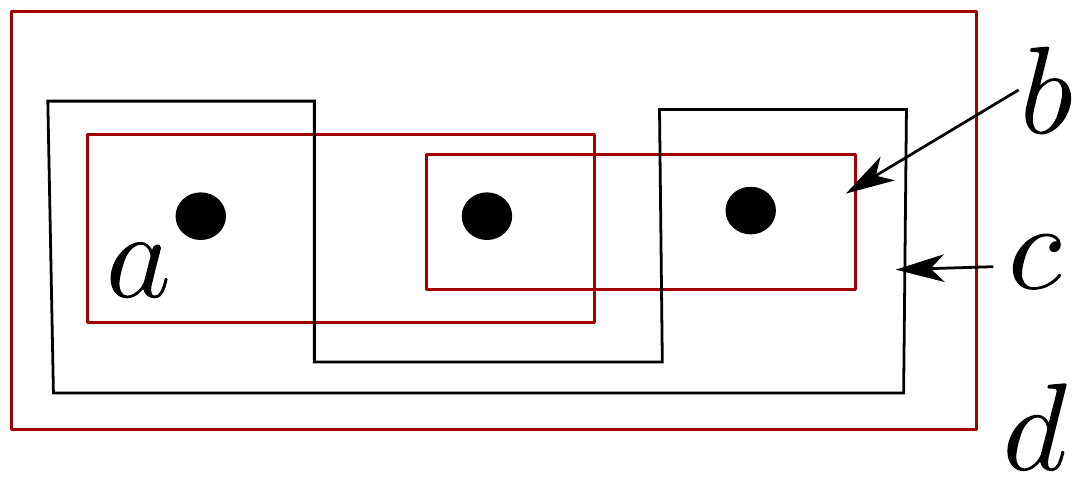}
  \caption{$D_n$}\label{CASE1}
\endminipage\hfill
\minipage{0.48\textwidth}
\centering
  \includegraphics[scale=0.4]{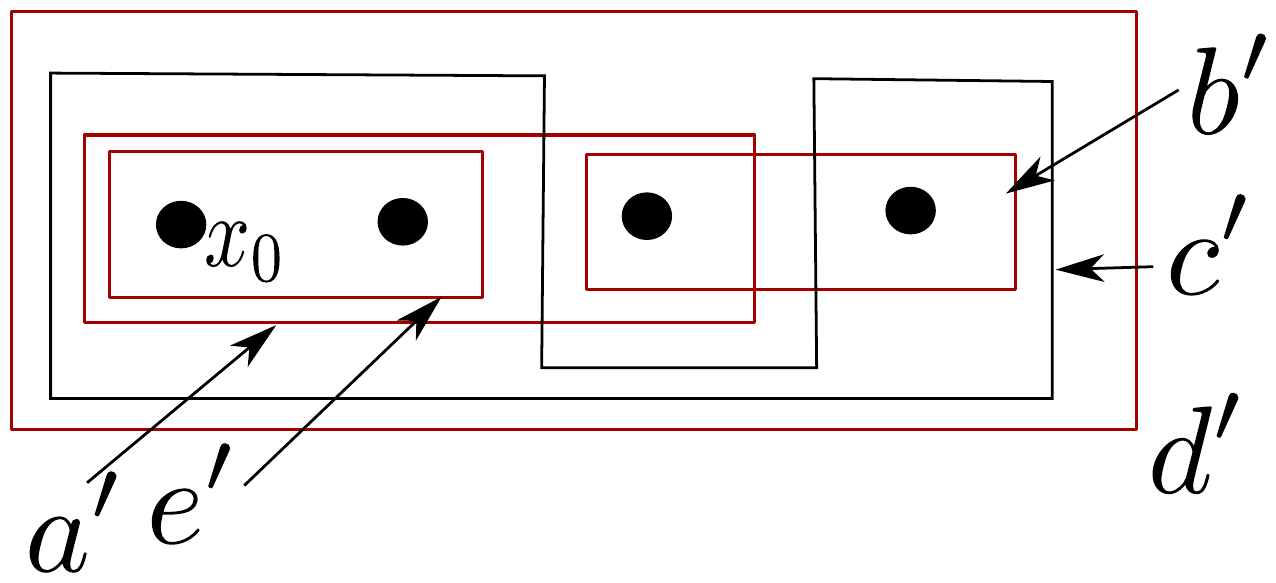}
    \caption{$D_{n+1}$}\label{CASE1(2)}
\endminipage\hfill
\end{figure}

Since $T_c$ and $T_d$ commute, $s(T_c)=T_{c'}$ and $s(T_d)=T_{d'}$ also commute. Therefore $s(T_a)s(T_b)=T_{c'}^{-1}T_{d'}$ is a multitwist by Lemma \ref{lift}. By Lemma \ref{2int}, we know that $i(b',a')=2$ as in Figure \ref{CASE1(2)}. Suppose that $b'$ does not surround $x_0$. By the lantern relation, $T_{a'}T_{b'}=T_{d'}T_{e'}T_{c'}^{-1}$. Since $s(T_d)$ and $s(T_c)$ are commuting multicurves, $s(T_d)s(T_c)^{-1}$ is multicurve as well. Since $a,b,c$ are basic curves, we know that $s(T_d)=T_{d'}T_{e'}$ and $s(T_c)=T_{c'}$. By the same reason, we have that $s(T_f)=T_{f'}$ in Figure \ref{CASE1(3)} and \ref{CAS}.

\begin{figure}[H]
\minipage{0.48\textwidth}
\centering
  \includegraphics[scale=0.4]{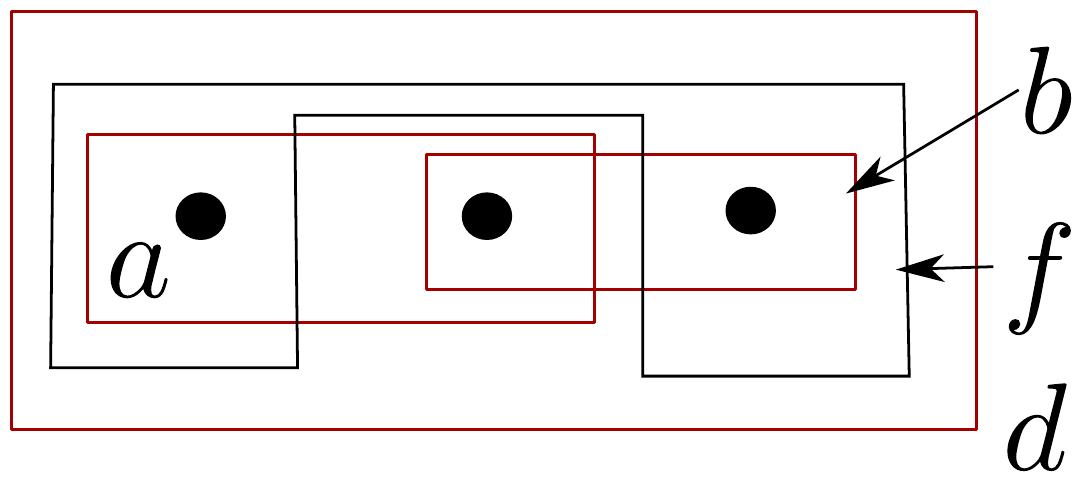}
  \caption{$D_n$}\label{CASE1(3)}
\endminipage\hfill
\minipage{0.48\textwidth}
\centering
  \includegraphics[scale=0.4]{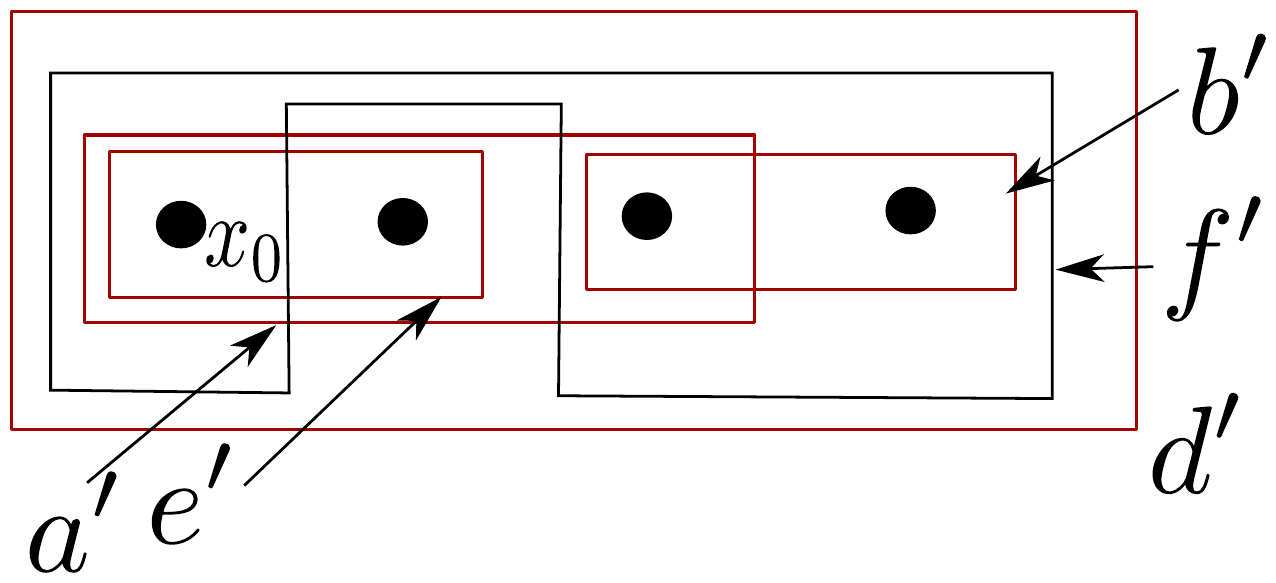}
    \caption{$D_{n+1}$}\label{CAS}
\endminipage\hfill
\end{figure}

In the following, we prove that $s(PB_n)$ preserves $b_{01}$. Under Notation \ref{notation} for $PB_n$, we have $a=a_{12}$, $b=a_{23}$, $c=a_{13}$ and $d=a_{123}$. We also have that $s(A_{12})=B_{012}$, $s(A_{23})=B_{23}$ and $s(A_{13})=B_{013}$. Since $A_{ij}$ generates $PB_n$, all we need to show is that $s(A_{ij})$ preserves $b_{01}$. Since CRS$(d)$ contains $b_{01}$, any curve disjoint from $d$ preserves $b_{01}$. We only need to consider the curves that intersect with $d$. Without loss of generality, we only need to show that $s(A_{14})$, $s(A_{24})$ and $s(A_{34})$ preserve $b_{01}$. By the assumption of Case 1, we only need to show that the CRS$(a_{14})$, CRS$(a_{24})$ and CRS$(a_{34})$ are disjoint from $b_{01}$. 

Since $i(a_{12},a_{34})=0$, we have that CRS$(a_{12})$ is disjoint from CRS$(a_{34})$, which means that CRS$(a_{34})$ is disjoint from $b_{01}$.  Since $s(T_f)=T_{f'}$ in Figure \ref{CASE1(3)} and Figure \ref{CAS}, CRS$(a_{24})$ is also disjoint from $b_{01}$. We have the following lantern relation:
\[
A_{13}A_{34}A_{14}=A_{134}.
\] 
The image of relation under lift $s$ is:
\[
B_{013}B_{34}s(A_{14})=s(A_{134}).
\] 
$A_{134}$ commutes with $A_{13}$ and $A_{34}$, thus CRS$(a_{134})$ is disjoint from $b_{013}$ and $b_{34}$. The only possible curves are $b_{01}$ and $b_{0134}$. If $s(A_{134})=B_{0134}$, we have another lantern relation in $D_{n+1}$:
\[
B_{013}B_{34}B_{014}=B_{0134}B_{01}.
\] 
This proves that $s(A_{14})=B_{014}B_{01}^{-1}$ preserving $b_{01}=e'$. If CRS$(a_{134})$ contains $b_{01}$, we also have that $s(A_{14})$ preserves $b_{01}=e'$. The case when $b'$ surrounds $x_0$ follows from the same argument.
\end{proof}

{\bf\boldmath Case 2: There exists a basic simple curve $a$ such that CRS$(a)$ has two curves and both are isotopic to $a$ after forgetting $x_0$.}
\begin{proof}[Proof of Case 2]
Let $b,c,d,e$ be curves in Figure \ref{15}. We have the lantern relation $T_bT_cT_d=T_eT_a$.
 \begin{figure}[H]
\minipage{0.48\textwidth}
\centering
  \includegraphics[scale=0.4]{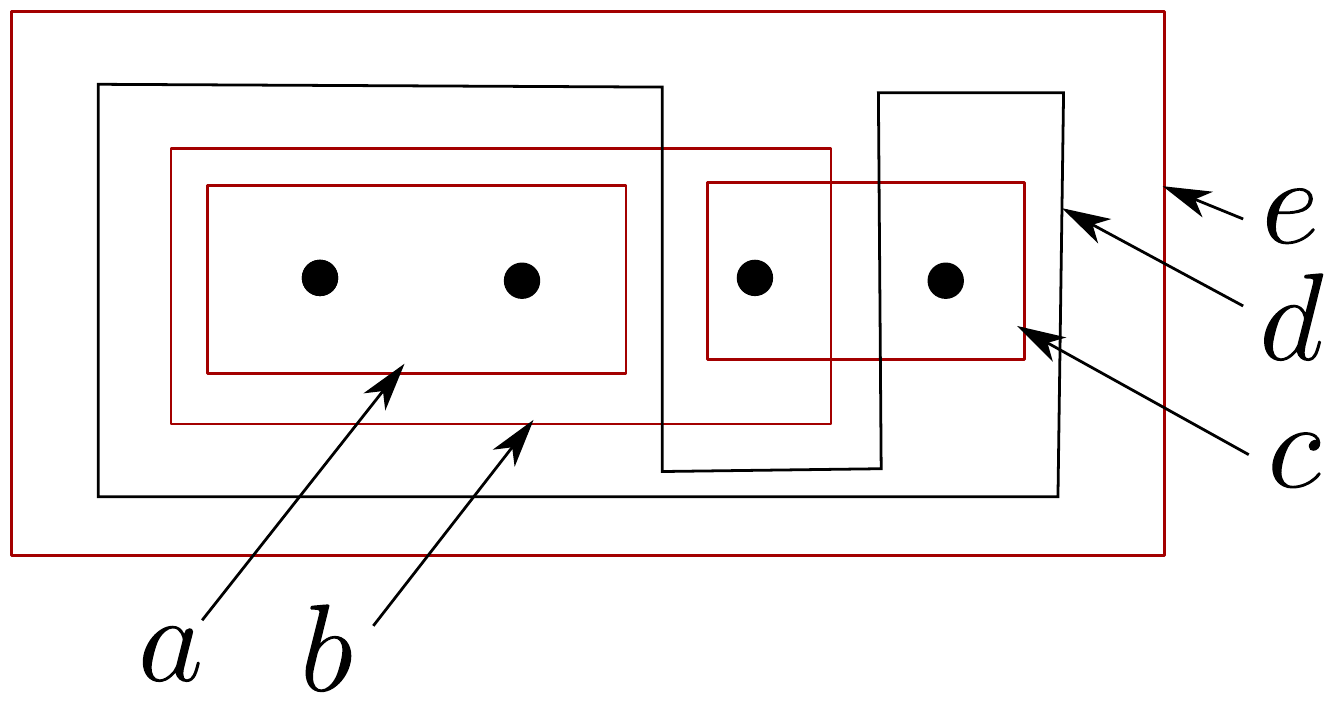}
  \caption{$D_n$}\label{15}
  \label{figure2}
\endminipage\hfill
\minipage{0.48\textwidth}
\centering
  \includegraphics[scale=0.4]{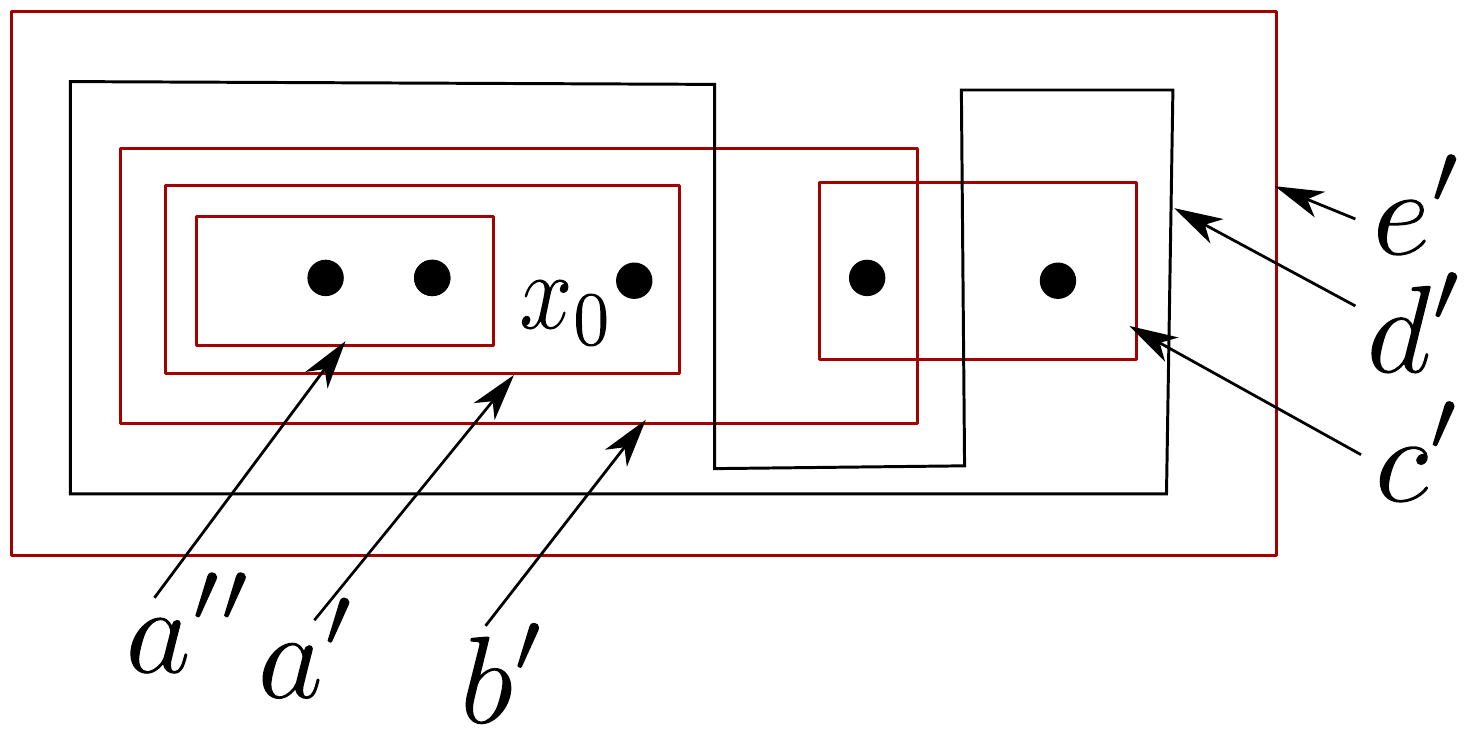}
  \caption{$D_{n+1}$}\label{16} \endminipage\hfill
\end{figure}

Since $b,c,d,e$ are disjoint from $a$, we have that CRS$(b)$, CRS$(c)$, CRS$(d)$ and CRS$(e)$ are disjoint from $\{a',a''\}$. Therefore $s(T_b)=T_{b'}$, $s(T_c)=T_{c'}$, $s(T_d)=T_{d'}$ and $s(T_e)=T_{e'}$ as in Figure \ref{16}. But we also have the lantern relation $T_{b'}T_{c'}T_{d'}=T_{e'}T_{a'}$. Thus $s(T_a)=T_{a'}$. This contradicts the assumption of Case 2.
\end{proof}
{\bf\boldmath Case 3: There exists a basic simple curve $a$ such that CRS$(a)$ has two curves $a',a''$ such that $a'$ is isotopic to $a$ and $a''$ is trivial after forgetting $x_0$, and $a'$ surrounds $a''$.} 
\begin{proof}[Proof of Case 3]
We arrange $a$ to be $a_{12}$ and $a',a''$ to be $b_{01},b_{012}$. Then we have $s(A_{12})=B_{012}B_{01}^k$ for $k\neq 0$ by Lemma \ref{lift}. Without loss of generality, we only need to show that CRS($a_{13}$) and CRS($a_{23}$) are disjoint from $b_{01}$. First of all, we have the following lantern relation:
\[
A_{123}A_{34}A_{124}=A_{12}A_{1234}.
\]

\begin{figure}[H]
\minipage{0.48\textwidth}
\centering
  \includegraphics[scale=0.5]{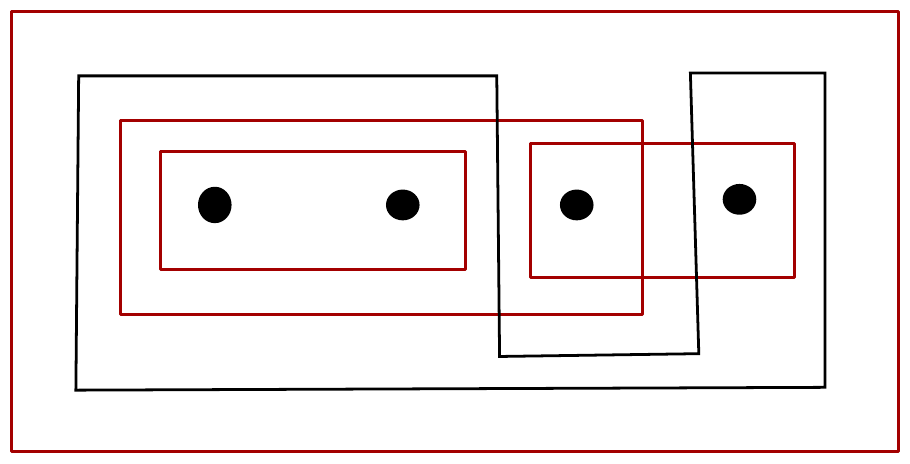}
  \caption{$D_n$}
\endminipage\hfill
\minipage{0.48\textwidth}
\centering
  \includegraphics[scale=0.5]{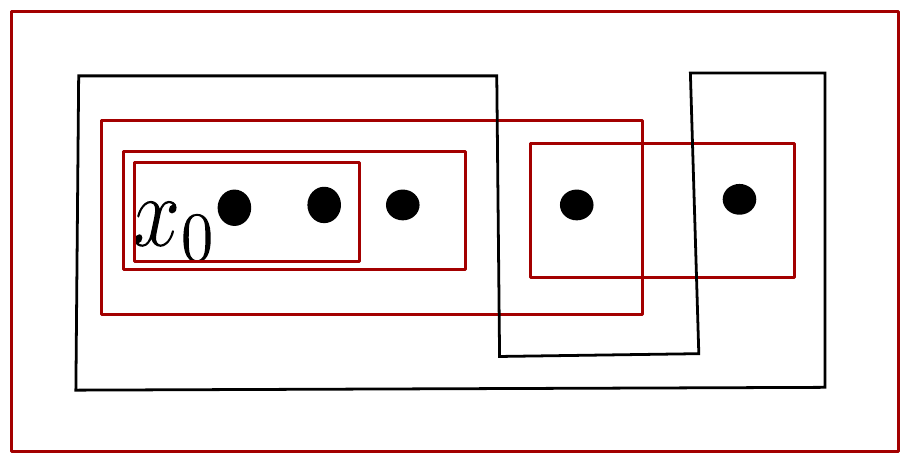}
  \caption{$D_{n+1}$}
\endminipage\hfill
\end{figure}

Since all of the curves above are disjoint from $a_{12}$, their canonical reduction systems are disjoint from $a_{12}'=b_{01}$ and $a_{12}''=b_{012}$. We have the lantern relation:
\[
B_{0123}B_{34}B_{0124}=B_{012}B_{01234}.
\]
Since $s(A_{12})=B_{012}B_{01}^k$, there exists at least one other curve in $a_{123}$, $a_{34}$, $a_{124}$, $a_{1234}$, whose canonical reduction system contains $b_{01}$. We break our discussion into the following four subcases depending on whether $b_{01}$ is an element in CRS$(A_{1234})$, CRS$(A_{123})$,  CRS$(A_{124})$ or  CRS$(A_{34})$, respectively.

\begin{figure}[H]
\minipage{0.44\textwidth}\centering
  \includegraphics[scale=0.5]{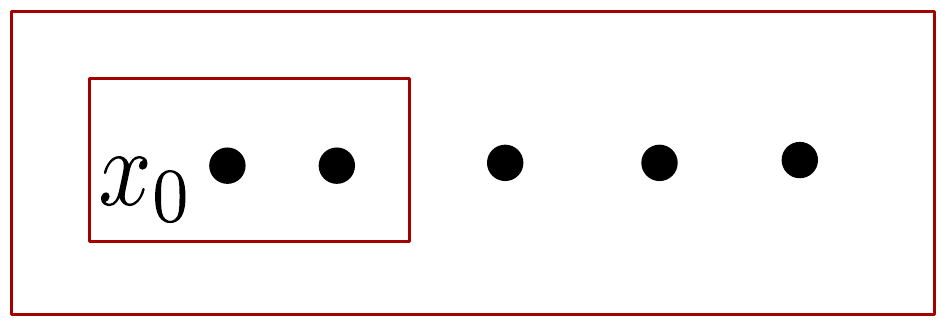}
  \caption{$b_{01}\in$ CRS$(A_{1234})$.}
  \label{figure2}
\endminipage\hfill
\minipage{0.44\textwidth}\centering
  \includegraphics[scale=0.5]{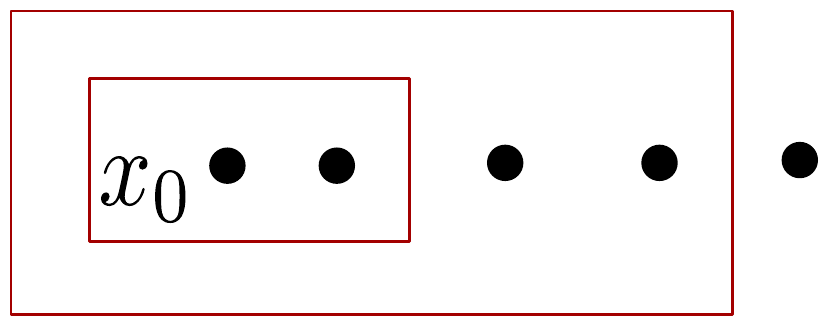}
  \caption{$b_{01}\in$ CRS$(A_{123})$.}
  \label{figure3}
  \endminipage\hfill
  \end{figure}
  \begin{figure}[H]
\minipage{0.44\textwidth}\centering
  \includegraphics[scale=0.5]{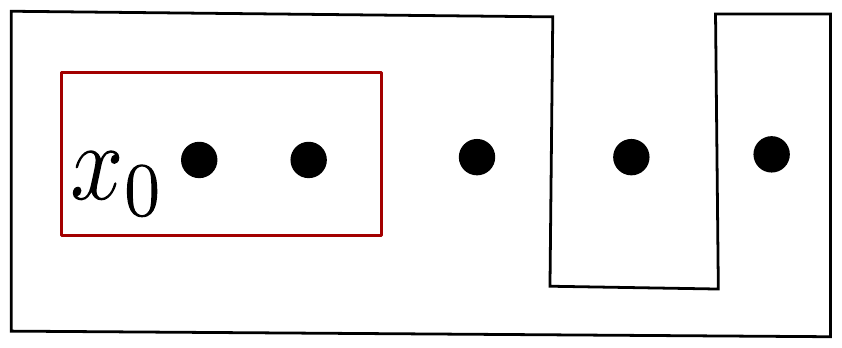}
  \caption{$b_{01}\in$ CRS$(A_{124})$.}
  \label{figure2}
\endminipage\hfill
\minipage{0.44\textwidth}\centering
  \includegraphics[scale=0.5]{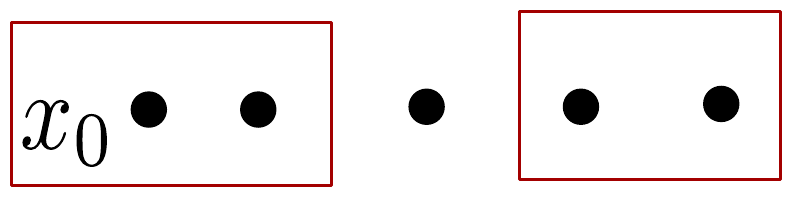}
  \caption{$b_{01}\in$ CRS$(A_{34})$.}
  \label{figure2}
\endminipage\hfill
\end{figure}
\noindent
{\bf\boldmath Subcase 1 and 2:} In the first two cases, it is clear that CRS$(a_{13})$ and CRS$(a_{23})$ are disjoint from $b_{01}$ because $a_{13}$ and $a_{23}$ are disjoint from $a_{123}$ and $a_{1234}$.\\
\\
{\bf\boldmath Subcase 3:} By $i(a_{14},a_{124})=0$, we have that $b_{014}\in \text{CRS}(a_{14})$ and $b_{01}$ does not intersect $\text{CRS}(a_{14})$. Since $i(a_{23},a_{14})=0$, we have that $b_{014}$ does not intersect CRS($a_{23}$). Suppose CRS$(a_{23})$ contains another curve $z$ that is trivial after forgetting $x_0$. Since $a_{23}$ is disjoint from $a_{123}$ and $a_{14}$, we have that $z$ has to be disjoint from $b_{014}$ and $b_{0123}$. The only possibility is that $z=b_{01}$.

Because of the disjointness of $a_{123}$ and $a_{12}$, we have that $s(A_{123})$ preserves CRS$(a_{12})$. This shows that $s(A_{123})$ preserves $b_{01}$. We have a lantern relation
\[
A_{12}A_{23}A_{13}=A_{123}
\]
After applying the homomorphism $s$ to the above relation, all of the above element except $s(A_{13})$ preserves $b_{01}$. Therefore $s(A_{13})$ fixes $b_{01}$.\\
\\
{\bf\boldmath Subcase 4: } Since $i(a_{234},a_{34})=0$, we have that $b_{234}\in \text{CRS}(a_{234})$. Since $i(a_{123},a_{12})=0$, we have that $b_{0123}\in \text{CRS}(a_{123})$. Therefore $b_{23}\in \text{CRS}(a_{23})$ and $\text{CRS}(a_{23})$ may contain another curve $z$ that is trivial after forgetting $x_0$. However $a_{23}$ is disjoint from $a_{123}$ and $a_{234}$, which implies that $z$ is disjoint from $b_{0123}$ and $b_{234}$. Therefore $z$ can only be $b_{01}$. By the same argument as Subcase 3, we know that $A_{13}$ also fixes $b_{01}$.\end{proof}
{\bf\boldmath Case 4: There exists a basic simple curve $a$ such that CRS$(a)$ has two curves $a',a''$ such that $a'$ is isotopic to $a$ and $a''$ is trivial after forgetting $x_0$, and $a'$ does not surround $a''$.} 
\begin{proof}[Proof of Case 4]
Let $a',a''$ be positioned into the following Figure \ref{figure10} such that $a'=b_{34}$ and $a''=b_{01}$. If a curve $c$ is disjoint from $a_{34}$, then $s(T_c)$ preserves $b_{01}$. Therefore without loss of generality, we only need to show that $s(A_{23})$ and $s(A_{13})$ preserve $b_{01}$.
\begin{figure}[H]
\centering
 \includegraphics[scale=0.5]{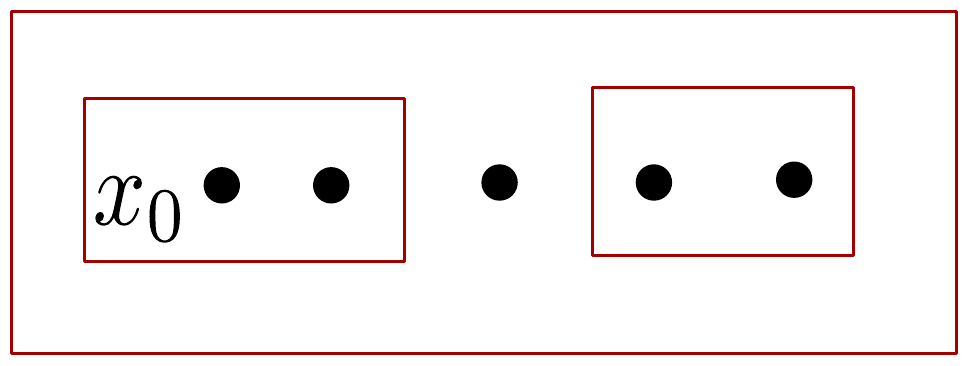}
 \caption{$a'=b_{34}$ and $a''=b_{01}$.}
  \label{figure10}
\end{figure}

Since $i(a_{12},a_{34})=0$, we have that $b_{012}\in\text{CRS}(a_{12})$. Since $i(a_{124},a_{12})=0$, we have that $b_{0124}\in\text{CRS}(a_{124})$. Possibly  CRS($a_{124}$) contains another curve $z$ that is trivial after forgetting $x_0$. However $b_{24}\in \text{CRS}(a_{24})$ because $b_{234}\in \text{CRS}(a_{234})$ and $i(a_{234},a_{24})=0$. Therefore, $z$ is disjoint from $b_{24}$ and $b_{012}$, which means $z=b_{01}$. By the same reason, we can prove that $b_{0123}\in \text{CRS}(a_{123})$ and $s(A_{123})$ preserves $b_{01}$.
We have the following lantern relation.
\begin{equation}
A_{123}A_{34}A_{124}=A_{12}A_{1234}.
\label{case44}
\end{equation}
Since $A_{34}=B_{34}B_{01}^k$ for nonzero $k$, therefore the canonical reduction system of one of the curves in the relation (\ref{case44}) contains $b_{01}$. The rest of the discussion is similar to Case 3 by doing a case study.
\end{proof}

\subsection{The proof of (1) of Corollary \ref{main11}}
\begin{proof}[\bf\boldmath Proof of (1) Corollary \ref{main11}]
Let $B_n=\pi_1(\text{PConf}_n(\mathbb{R}^2)/\Sigma_n)$ and $B_{n,1}=\pi_1(\text{PConf}_{n+1}(\mathbb{R}^2)/\Sigma_n)$. The fiber bundle $F_n(S)$ gives the first line of the following commutative diagram:
\[
\xymatrix{
1\ar[r]&     PB_{n+1}\ar[r]\ar[d]^{f_{n}(\mathbb{R}^2)_*}&    B_{n,1}\ar[r]\ar[d]^{F_{n}(\mathbb{R}^2)_*}&          \Sigma_n   \ar[r]\ar[d]^{=} &1\\
1\ar[r]&      PB_n\ar[r]&                             B_n\ar[r]&                                   \Sigma_n  \ar[r]        &   1  .  }
\]

Every splitting of $F_{n}(\mathbb{R}^2)_*$ induces a splitting of $f_{n}(\mathbb{R}^2)_*$. Therefore, we only need to study the extension of a splitting of $f_{n}(\mathbb{R}^2)_*$ to a splitting of $F_{n}(\mathbb{R}^2)_*$. Let $\phi:B_n\to B_{n,1}$ be a splitting of $F_{n}(\mathbb{R}^2)_*$. Let $x\in PB_n$ and $e\in B_n$. We have that 
\[
\phi(exe^{-1})=\phi(e)\phi(x)\phi(e)^{-1}.
\]
Denote by $C_e$ the conjugation of $e$ on $PB_n$. Therefore, we have the following diagram:
\begin{equation}
\xymatrix{
PB_n\ar[r]^{C_e}\ar[d]^{\phi|_{PB_n}}  &     PB_n\ar[d]^{\phi|_{PB_n}} \\
PB_{n+1}\ar[r]^{C_{\phi(e)}}&        PB_{n+1}. }
\label{DG}
\end{equation}

By Theorem \ref{PB}, there are two possibilities of $\phi|_{PB_n}$: \\
\\
(1) $\phi$ fixes a simple closed curve $c$ surrounding $\{x_k,x_0\}$\\
\\
(2) $\phi$ fixes a simple closed curve $c$ surrounding $\{x_1,...,x_n\}$.\\

We claim that $\phi|_{PB_n}$ fixes a simple closed curve surrounding $\{x_1,...,x_n\}$. To prove this claim, we assume the opposite that $\phi|_{PB_n}$  fixes a simple closed curve $c$ surrounding $\{x_k,x_0\}$. There exists an element $e\in B_n$ such that $e$ permutes punctures $k$ and $j\neq k$. Since $c$ is the only curve that $\phi(PB_n)$ fixes, we have that $c$ is the only curve that $\phi(C_e(PB_n))=PB_n$ fixes, which contradicts that $C_{\phi(e)}(\phi(PB_n))$ also fixes $\phi(e)(c)$ surrounding $\{x_j,x_0\}$. Therefore, $\phi|_{PB_n}$ fixes a simple closed curve surrounds $\{x_1,...,x_n\}$. In this case, the section is adding a point at infinity.
\end{proof}

\section{The case when $S$ is the 2-sphere $S^2$}
In this subsection we give a construction of sections of the fiber bundle $f_n(S^2)$.

\subsection{Nonexistence of a continuous section for $n=2$}
We prove a more general result on the sections of the fiber bundle $f_n(S^2)$ for $n=2$. Let $S^{2k}$ be $2k$-dimensional sphere for $k>0$ integer.  Let $x_1,x_2$ be two distinct points in $S^{2k}$. The following is classical; see \cite[Chapter 3]{FH}.
\begin{prop}
The following fiber bundle
\[
S^{2k}-\{x_1,x_2\}\to \text{\normalfont PConf}_3(S^{2k})\xrightarrow{f_2(S^{2k})} \text{\normalfont PConf}_2(S^{2k})
\]
does not have a continuous section.
\end{prop}
\begin{proof}
Suppose that there is a continuous map $s:\text{PConf}_2(S^{2k})\to \text{PConf}_{3}(S^{2k})$ such that $f_2(S^{2k})\circ s=\text{identity}$. Then after post-composing with a forgetful map to the last coordinate, we obtain a map $f: \text{PConf}_2(S^{2k})\to S^{2k}$. We denote by $p_i:\text{PConf}_2(S^{2k})\to S^{2k}$ the projection to the $i$th component. Let
\[
g_i:\text{PConf}_2(S^{2k}) \xrightarrow{(f,p_i)} \text{PConf}_2(S^{2k})\subset S^{2k}\times S^{2k}.
\]

Let $\triangle\subset S^{2k}\times S^{2k}$ be the diagonal subspace in the product. Let $[\triangle]\in H^{2k}(S^{2k}\times S^{2k},\mathbb{Q})$ be the Poincar\'e dual of $\triangle$. By the Thom isomorphism, there is an exact sequence for the computation of cohomology:
\[
0\to \mathbb{Q}\xrightarrow{\text{diagonal}} H^{2k}(S^{2k}\times S^{2k},\mathbb{Q})\to H^{2k}(S^{2k}\times S^{2k}-\triangle;\mathbb{Q})\to 0.
\]
Let $c\in H^{2k}(S^{2k};\mathbb{Q})$ be the fundamental class and $c_i=p^*_i(c)$. The image of diagonal is the Thom class $c_1+c_2$. Therefore in $H^{2k}(S^{2k}\times S^{2k}-\triangle;\mathbb{Q})$, we have $c_1+c_2=0$. This means that 
\[
H^{2k}(S^{2k}\times S^{2k}-\triangle;\mathbb{Q})=\mathbb{Q}c_1.
\]

Suppose that $f^*(x)=kc_1$ for $k$ an integer. Therefore we will have $g_i^*([\triangle])=kc_1+c_i$. Since the image of $g_i$ misses the diagonal $\triangle$, we have that $g_1^*([\triangle])=kc_1+c_1=0$ and $g_2^*([\triangle])=kc_1+c_2=kc_1-c_1=0$. Since $c_1$ is a generator of $H^{2k}(S^{2k}\times S^{2k}-\triangle;\mathbb{Q})$, we have that $k+1=0$ and $k-1=0$. These two formulas cannot be satisfied at the same time.
\end{proof}

\subsection{Constructing sections when $n>2$}
Define
\[\text{PConf}_{n,k}(S^2)=\{(v_1,x_1,...,x_n)|x_1,...,x_n \text{ be } n \text{ points in } S^2 \text{ and } v_k \text{ be a unit vector at } x_k\}.
\]
This is the total space of a circle bundle by forgetting the vector:
\begin{equation}
S^1\to \text{PConf}_{n,k}(S^2)\to \text{PConf}_{n}(S^2).
\label{vectorsphere}
\end{equation}

\begin{prop}
For $n>2$, the fiber bundle (\ref{vectorsphere}) is a trivial bundle.
\label{trivial}
\end{prop}
\begin{proof}
$S^1$-bundle is classified by Euler class, i.e. a second cohomology class of the base. We investigate $H^2(\text{PConf}_{n}(S^2);\mathbb{Z})$ first. There is a graded-commutative $\mathbb{Q}$-algebra $[G_{ij}]$ defined in \cite[Theorem 1]{totaro}, where the degree of the generators $G_{ij}$ is 1. By Totaro \cite[Theorem 1]{totaro}, there is a spectral sequence $E_2^{p,q}=H^p((S^2)^n;\mathbb{Q})[G_{ij}]^q$ converging to $H^*(\text{PConf}_n(S^2);\mathbb{Q})$. Since we only compute $H^2$, the differential involved is $d_2:E_2^{0,1}=H^0(S^2;\mathbb{Q})[G_{ij}]\to E_2^{2,0}=H^2(S^2;\mathbb{Q})$. Let $[\triangle_{ij}]\in H^2(S^2;\mathbb{Q})$ be the Poincar\'e dual of $\triangle_{ij}\subset S^2$. By \cite[Theorem 2]{totaro}, the differential $d_2(G_{ij})=[\triangle_{ij}]$. Let $p_i:(S^2)^n\to S^2$ be the projection to the $i$th coordinate and $[S^2]\in H^2(S^2;\mathbb{Z})$ be the generator of $H^2(S^2;\mathbb{Z})$. Therefore we have that
\[
H^2(\text{PConf}_{n}(S^2);\mathbb{Z})=\bigoplus_{i=1}^n\mathbb{Z}p_i^*[S^2]/(p_i^*[S^2]+p_j^*[S^2])\cong \mathbb{Z}/2,
\]
which is generated by $p_k^*[S^2]$ and we have that $2p_k^*[S^2]=0$. The circle bundle \eqref{vectorsphere} is induced from the circle bundle 
\begin{equation}
S^1\to \text{PConf}_{1,1}(S^2)\to S^2
\label{UTBS}
\end{equation}
by the projection to the $k$th coordinate. The bundle \eqref{UTBS} is the unit tangent bundle over $S^2$. Since the Euler characteristic of $S^2$ is 2, the Euler class of \eqref{UTBS} is $eu=2[S^2]\in H^2(S^2;\mathbb{Z})$. Therefore the Euler class of \eqref{vectorsphere} is $p_k^*[eu]=2p_k^*[S^2]=0\in H^2(\text{PConf}_{n}(S^2);\mathbb{Z})$.
\end{proof}

Equip $S^2$ with the spherical metric; i.e. the metric that is induced from the standard embedding $S^2\subset \mathbb{R}^3$. Let 
\[
\epsilon(x_1,...,x_n)=\frac{1}{2} \text{min}_{1\le i<j\le n}\{d(x_i,x_j)\}.
\]
Set $x_0$ to be the image of the $v_k$-flow at time $\epsilon$ from $x_k$; that is
\[
em_{n,k}(S^2):\text{PConf}_{n,k}(S^2)\hookrightarrow\text{PConf}_{n+1}(S^2)
\] 

Composing a continuous section $s:\text{PConf}_{n}(S^2)\to \text{PConf}_{n,k}(S^2)$ of the fiber bundle (\ref{vectorsphere}) with $em_{n,k}(S^2)$ gives a section of the fiber bundle $f_n(S^2)$. 
\begin{defn}[\bf\boldmath Adding a point near $x_k$]
We denote by Add$_{n,k}(S^2)$ the collection of sections of $f_n(S^2)$ consisting of compositions of a section of \eqref{vectorsphere} with $em_{n,k}(S^2)$.
\end{defn}
Notice that there are infinitely many homotopy classes of sections in Add$_{n,k}(S^2)$ and they are classified by sections of \eqref{vectorsphere}.\\
\\
\noindent
{\bf\boldmath A special section for $n=3$. } Since there is a unique Mobius transformation $\phi(x_1,x_2,x_3)$ that transforms $(0,1,\infty)$ to any ordered three points $(x_1,x_2,x_3)$. we have that 
\[
\text{PConf}_3(S^2)      \xrightarrow[\approx]{\phi} \text{PSL}(2,\mathbb{C}).
\]
We can assign any new point $x_0=\phi(x_1,x_2,x_3))(a)$ such that $a\neq 0,1,\infty$.

\subsection{The proof of (2) of Theorem \ref{main}}
In this subsection we prove (2) of Theorem \ref{main}. Let $S_{0,n}$ a sphere with $n$ punctures. Let Diff$(S_{0,n})$ be the orientation-preserving diffeomorphism group of $S_{0,n}$ fixing the $n$ punctures pointwise. While the following is surely known to experts, we could not find this statement or a proof in the literature. I am thus incluing it for completeness. We believe that it follows from Earle-Eells \cite[Theorem 1]{EE} in the punctured case.
\begin{prop}
\label{BD}For $n>2$, we have that 
\[
\text{\normalfont BDiff}(S_{0,n})\cong K(\text{\normalfont PMod}(S_{0,n}),1)\]
\end{prop}
\begin{proof}
We only need to prove that the homotopy group $\pi_k(\text{Diff}(S_{0,n}))=0$ for $k>0$. For $n=0$, by Smale \cite[Theorem A]{MR0112149}, Diff$(S^2)\simeq \text{SO}(3)$. By fiber bundle
\begin{equation}
\text{Diff}(S_{0,n+1})\to \text{Diff}(S_{0,n})\to S_{0,n},
\label{sph}
\end{equation}
we deduce that Diff$(S_{0,1})\simeq \text{SO}(2)$ and Diff$(S_{0,2})\simeq \text{SO}(2)$. The long exact sequence of homotopy groups of the fiber bundle (\ref{sph}) is 
\[
1\to \pi_1(\text{Diff}(S_{0,3}))\to \pi_1( \text{Diff}(S_{0,2}))\to \pi_1(S_{0,2})\to \text{PMod}_{0,3}\to \text{PMod}_{0,2}\to 1.
\]
However we know that PMod$_{0,3}=1$ (see \cite[Proposition 2.3]{BensonMargalit}), we get that $\pi_1(\text{Diff}(S_{0,3}))=0$ and also $\pi_i(\text{Diff}(S_{0,3}))=0$ for $i>1$. The other cases are the same.
\end{proof}
Let $PB_{n}(S^2)=\pi_1(\text{PConf}_n(S^2))$. Now we are ready to prove (2) of Theorem \ref{main}.
\begin{proof}[\bf\boldmath Proof of (2) of Theorem \ref{main}]
Let 
\[
S_{0,n+1}\to \text{UDiff}(S_{0,n+1})\xrightarrow{u_{n+1}} \text{BDiff}(S_{0,n+1})\]
be the universal $S_{0,n+1}$-bundle in the sense that any $S^2$ bundle with $n+1$ sections 
\[S_{0,n+1}\to E\to B\]
is the pullback from $u_{n+1}$ by a continuous map $f:B\to \text{BDiff}(S_{0,n+1})$. By Proposition \ref{BD}, $\text{BDiff}(S_{0,n+1})\cong K(\text{PMod}(S_{0,n+1}),1)$. This means that $\text{UDiff}(S_{0,n+1})$ is also a $K(\pi,1)$-space. Therefore $S_{0,n+1}$-bundles are determined by their monodromy representations and the sections of an $S_{0,n+1}$-bundle are also determined by the maps on fundamental groups. A splitting of the following exact sequence gives us a section of the fiber bundle $f_n(S^2)$.
\[
1\to F_n\to PB_{n+2}(S^2)\xrightarrow{f_{n}(S^2)_*} PB_{n+1}(S^2)\to 1.
\]

We have the following diagram:
\[
\xymatrix{
S_{0,n+1}\ar[r]\ar[d]& \text{PConf}_{n+1}(S_{0,1})\ar[r]^{f_{n}(S_{0,1})}\ar[d] &\text{PConf}_{n}(S_{0,1})\ar[d]\\
S_{0,n+1}\ar[r] &\text{PConf}_{n+2}(S^2)\ar[r]^{f_{n+1}(S^2)}\ar[d]^{p_{n+1}} &\text{PConf}_{n+1}(S^2)\ar[d]^{p_{n+1}}  \\
&S^2\ar[r]&S^2.
}
\]

By the long exact sequence of homotopy groups of the fiber bundle 
\[\text{PConf}_n(S_{0,1})\to \text{PConf}_{n+1}(S^2)\to S^2,
\]
we have that $PB_{n+1}(S^2)=PB_n/Z$ where $Z$ denotes the center of $PB_n$ and is generated by the Dehn twist about the boundary of $D_n$; see \cite[Page 247]{BensonMargalit}. Therefore a section of $f_{n}(S_{0,1})$ induced from a section of $f_{n+1}(S^2)$ satisfies that $f_{n}(S_{0,1})_*$ maps the center to the center.

Since $S_{0,1}\approx \mathbb{R}^2$, the section problem for $f_{n}(S_{0,1})$ has been fully discussed in Section 2. Every section of $f_{n+1}(S^2)$ induces a section of $f_{n}(\mathbb{R}^2)$, thus we could use the classification of sections of $f_{n}(\mathbb{R}^2)$ to study the sections of $f_{n+1}(S^2)$. Let $s:PB_{n+1}(S^2)\to PB_{n+2}(S^2)$ be a splitting of $f_{n+1}(S^2)_*$ such that $f_{n+1}(S^2)_*\circ s=id$. By (1) of Theorem \ref{main}, we break the discussion into the following two cases according to the sections of $f_n(S_{0,1})$. \\
\\
{\bf\boldmath Case 1: the section of $f_{n}(S_{0,1})$ is adding a point near $x_k$.} In this case, $s(PB_{n+1}(S^2))$ fixes a curve $c$ around $\{x_0,x_k\}$. Then the image lies in the stabilizer of $c$. The stabilizer of $c$ in $\text{PMod}(S_{0,n+2})$ is $\text{PMod}(D_{n})\cong PB_n$. The boundary of $D_n$ is $c$ surrounding $\{x_1,...,x_{k-1},x_{k+1},...,x_{n+1}\}$. On the other hand by Proposition \ref{trivial} the circle bunlde
\[
S^1\to \text{PConf}_{n,k}(S^2)\to \text{PConf}_{n}(S^2)
\]
is trivial, we have that 
\[
\pi_1(\text{PConf}_{n+1,k}(S^2))\cong\mathbb{Z}\times \pi_1(\text{PConf}_{n+1}(S^2))\cong \mathbb{Z}\times PB_{n+1}(S^2)\cong PB_n.
\]
The last isomorphism is coming from the splitting of the following exact sequence; see \cite[Page 252]{BensonMargalit}.
\[
1\to Z\to PB_n\to PB_n/Z\to 1.
\]

Since the $\mathbb{Z}$ component of $\pi_1(\text{PConf}_{n+1,k}(S^2))$ is mapped to the Dehn twist about a curve $d$ surrounding $\{x_0,x_k\}$, it means that $d$ also surrounds $\{x_1,...,x_{k-1},x_{k+1},...,x_{n+1}\}$. Therefore we have that $f_{n+1}(S^2)$ is adding a point near $x_k$.\\
\\
{\bf\boldmath Case 2: the section of $f_{n}(S_{0,1})$ is adding a point near $\infty$.} In this case, $s(PB_{n+1}(S^2))$ fixes a curve $c$ around $\{x_1,...,x_n\}$. Then the image lies in the stabilizer of $c$. The stabilizer of $c$ in $\text{Mod}(S_{0,n+2})$ is $\text{PMod}(D_{n})\cong PB_n$. The boundary of $D_n$ is $c$ surrounding $\{x_1,...,x_n\}$. On the other hand by Proposition \ref{trivial} the circle bunlde
\[
S^1\to \text{PConf}_{n+1,{n+1}}(S^2)\to \text{PConf}_{n+1}(S^2).
\]
is trivial, we have that 
\[
\pi_1(\text{PConf}_{n+1,n+1}(S^2))\cong\mathbb{Z}\times \pi_1(\text{PConf}_{n+1}(S^2))\cong \mathbb{Z}\times PB_{n+1}(S^2)\cong PB_n.
\]
Since the $\mathbb{Z}$ component of $\pi_1(\text{PConf}_{n+1,k}(S^2))$ is mapped to the Dehn twist about a curve $d$ surrounding $\{x_0,x_{n+1}\}$, it means that $d$ also surrounds $\{x_1,...,x_{n}\}$. Therefore we have that $f_{n+1}(S^2)$ is adding a point near $x_{n+1}$.
\end{proof}

\subsection{The unordered case}

\begin{proof}[\bf Proof of (2) of Corollary \ref{main11}]
By the same argument as the proof of (1) of Corollary \ref{main11}, we show that none of the sections of 
\[
f_{n}(S^2): \text{PConf}_{n+1}(S^2)\to \text{PConf}_n(S^2)
\]
can be extended to a section of 
\[
F_{n}(S^2): \text{PConf}_{n+1}(S^2)/\Sigma_n\to \text{PConf}_n(S^2)/\Sigma_n.
\]
\end{proof}

\subsection{The exceptional cases}\label{exception}
For the special cases $n=3$, we have the following classification.
\begin{thm}[\bf Classification of sections of $f_3(S^2)$ and  $F_3(S^2)$]
There is a unique section for the fiber bundle $f_3(S^2)$ up to homotopy. There is no section for the bundle $F_3(S^2)$.
\end{thm}
\begin{proof}
By Proposition \ref{BD}, we have that BDiff$(S_{0,3})\cong K(\text{PMod}(S_{0,3}),1)$. Since PMod$(S_{0,3})=1$, the classifying space BDiff$(S_{0,3})$ is contractible. Therefore every $S^2$-bundle with $3$ sections is a trivial bundle. Thus $f_3(S^2)$ is a trivial bundle. Therefore, a section of $f_3(S^2)$ is determined by a map PConf$_3(S^2)\to S_{0,3}$. Since $S_{0,3}\cong K(F_2,1)$, a map PConf$_3(S^2)\to S_{0,3}$ up to homotopy is determined by Hom$(PB_3(S^2),F_2)$ up to conjugation. However $PB_3(S^2)=PB_2/Z=1$ implying that Hom$(PB_3(S^2),F_2)=1$. Therefore, there is a unique section up to homotopy. For the unordered case $F_3(S^2)$, let Mod$(S_{0,3,1})$ be the mapping class group of $S^2$ fixing a set of 3 points and a set of 1 point. There is an exact sequence
\begin{equation}
1\to \text{PMod}(S_{0,4})\to \text{Mod}(S_{0,3,1})\to \Sigma_3 \to 1.
\label{Sigma3}
\end{equation}
Since $\pi_1(\text{PConf}_3(S^2)/\Sigma_3)\cong \Sigma_3$ and $\pi_1(S_{0,3})\cong \text{PMod}(S_{0,4})$, we have that $\text{Mod}(S_{0,3,1})=\pi_1(\text{PConf}_4(S^2)/\Sigma_3)$. Therefore, the section of $F_3(S^2)$ is the determined by the splittings of the exact sequence \eqref{Sigma3}.

Let $\overline{\text{Diff}}(S_{0,3,1})$ be the orientation-preserving diffeomorphism group of $S^2$ fixing a set of 3 points and a set of 1 point. By definition there is a map $\rho: \overline{\text{Diff}}(S_{0,3,1})\to \text{Mod}(S_{0,3,1})$ which induces isomorphism on $\pi_0$. A version of the Nielsen Realisation Theorem (e.g. \cite[Theorem 7.2]{BensonMargalit} and \cite{Wolpert}) tells us that a finite subgroup of Mod$(S_{0,3,1})$ has a lift to $\overline{\text{Diff}}(S_{0,3,1})$. However every finite subgroup of $\overline{\text{Diff}}(S_{0,3,1})$ is cyclic because $\overline{\text{Diff}}(S_{0,3,1})$ fixes a point. Therefore every finite subgroup of Mod$(S_{0,3,1})$ is cyclic. Since $\Sigma_3$ is noncyclic, \eqref{Sigma3} does not split.
\end{proof}
For the special cases $n=4$, we have the following classification.
\begin{thm}[\bf Classification of sections of $f_4(S^2)$]
The sections of fiber bundle $f_4(S^2)$ correspond to the splittings of the exact sequence
\[1\to F_3 \to PB_5(S^2)\xrightarrow{f_4(S^2)_*} F_2\to 1\]
up to conjugation.
\end{thm}
\begin{proof}
We have the following Birman exact sequence; see \cite[Theorem 4.6]{BensonMargalit}.
\[
1\to \pi_1(S_{0,3})\to PB_4(S^2)\xrightarrow{f_3(S^2)_*} PB_3(S^2)\to 1.\]
Since $PB_3(S^2)=1$, we have that $PB_4(S^2)=\pi_1(S_{0,3})\cong F_2$. By Proposition \ref{BD}, the sections of $f_4(S^2)$ is determined by the splittings of the following Birman exact sequence up to conjugaction.
\[
1\to \pi_1(S_{0,4})\to PB_5(S^2)\xrightarrow{f_4(S^2)_*} PB_4(S^2)\to 1.\qedhere
\] 
\end{proof}
\section{The case when $S=S_g$ a closed surface of genus $g>1$}
In this section, we prove Theorem \ref{main}(3).  Let $S_g^n$ be the product of $n$ copies of $S_g$. There is a natural embedding $\text{PConf}_n(S_g)\subset S_g^n$. Let $p_i:\text{PConf}_n(S_g)\to S_g$ be the projection onto the $i$th component. Denote by $\triangle_{ij}\approx S_g^{n-1}\subset S_g^n$ the $ij$th diagonal subspace of $S_g^n$; i.e., $\triangle_{ij}$ consists of points in $S_g^n$ such that the $i$th and $j$th coordinates are equal. Let $H_i:=p_i^*H^1(S_g;\mathbb{Q})$ and let $[S_g]$ be the fundamental class in $H^2(S_g;\mathbb{Q})$. Now, we display the computation of $H^*(\text{PConf}_n(S_{g});\mathbb{Q})$ from \cite{lei1}. 
\begin{lem}
(1) For $g>1$ and $n>0$,
\begin{equation}
H^1(\text{\normalfont PConf}_n(S_{g});\mathbb{Q})\cong H^1(S_{g}^n;\mathbb{Q})\cong \bigoplus_{i=1}^{n}H_i.
\label{EX1}
\end{equation}
(2)We have an exact sequence
\begin{equation}
1\to \oplus_{1\le i<j\le n}\mathbb{Q}[G_{ij}] \xrightarrow{\phi} H^2(S_{g}^n;\mathbb{Q})\cong \bigoplus_{i=1}^{n}\mathbb{Q}p_i^*[S_g]\oplus \bigoplus_{i\neq j} H_i\otimes H_j\xrightarrow{Pr} H^2(\text{PConf}_n(S_{g});\mathbb{Q}),
\label{EX2}
\end{equation}
where $\phi(G_{ij})=[\triangle_{ij}]\in H^2(S_{g}^n;\mathbb{Q})$ is the Poincar\'e dual of the diagonal $\triangle_{ij}$.
\label{1}
\end{lem}
\begin{proof}
See \cite[Lemma 3.1]{lei1}.\end{proof}

Let $\{a_k,b_k\}_{k=1}^g$ be a symplectic basis for $H^1(S_g;\mathbb{Q})$. For $1\le i,j \le m$, we denote 
\[M_{i,j}=\sum_{k=1}^{n} p_i^*a_k \otimes p_j^*b_k-p_i^*b_k\otimes p_j^*a_k.\]

\begin{lem}
The diagonal element $[\triangle_{ij}]=p_i^*[S_{g}]+p_j^*[S_{g}]+M_{ij}\in \bigoplus_{i=1}^{n}\mathbb{Q}p_i^*[S_g]\oplus \bigoplus_{i\neq j} H_i\otimes H_j\cong H^2(S_{g}^n;\mathbb{Q})$. 
\label{2}
\end{lem}
\begin{proof}
See \cite[Lemma 3.2]{lei1}.\end{proof}

The following lemma is the classification of homomorphisms $\pi_1(\text{PConf}_n(S_g))\to \pi_1(S_g)$ from \cite{lei1}.
\begin{thm}[{\bf \boldmath The classification of homomorphisms $\pi_1(\text{PConf}_n(S_g))\to \pi_1(S_g)$}]
Let $g>1$ and $n>0$. Let $R: \pi_1(\text{PConf}_n(S_g))\to \pi_1(S_g)$ be a homomorphism. The followings hold:\\
\\
(1)If $R$ is surjective, then $R=A\circ p_{i*}$ for some $i$ and $A$ an automorphism of $\pi_1(S_g)$.\\ 
\\
(2)If Image$(R)$ is not a cyclic group, the homomorphism $\pi_1(\text{PConf}_n(S_g))\to \pi_1(S_g)$ factors through $p_{i*}$ for some $i$.
\label{3}
\end{thm}

\begin{proof}
See \cite[Theorem 1.5]{lei1}.
\end{proof}
Now, we are ready to prove (3) of Theorem \ref{main}.
\begin{proof}[\bf\boldmath Proof of (3) of Theorem \ref{main}]
Suppose that there is a map $s:\text{PConf}_n(S_g)\to \text{PConf}_{n+1}(S_g)$ such that $f_{n}(S_g)\circ s=\text{identity}$. Then after post-composing with a forgetful map of the last coordinate, we obtain a map $f: \text{PConf}_n(S_g)\to S_g$. We denote 

\[
g_i:\text{PConf}_n(S_g) \xrightarrow{(f,p_i)} \text{PConf}_2(S_g)\subset S_g\times S_g
\]
Let $\triangle\subset S_g\times S_g$ be the diagonal subspace and $[\triangle]\in H^2(S_g\times S_g;\mathbb{Q})$ be the Poincar\'e dual of $\triangle$. Let $f^*:H^1(S_g)\to H^1(\text{PConf}_n(S_g))$ and $f_*:\pi_1(\text{PConf}_n(S_g))\to \pi_1(S_g)$ be the induced map on cohomology and the fundamental groups. By Lemma \ref{3}, either $f_*$ factors though a forgetful map $p_{i*}$ or Image$(f_*)\cong \mathbb{Z}$. We break the proof into two cases according to the image of $f_*$.\\
\\
{\bf\boldmath Case 1: Image$(f_*)\cong \mathbb{Z}$.} There are two subcases:\\
\\
(1) If $f^*=0$, then $g_i^*([\triangle])=p_i^*[S_{g}]\neq 0$. This contradicts the fact that the image of $g_i$ misses $\triangle$.\\
\\
(2) If $f^*\neq 0$, then Im$f^*\cong \mathbb{Z}$ because $f_*$ has image $\mathbb{Z}$ on the fundamental groups. We assume that there exists a symplectic basis $\{a_k,b_k\}_{k=1}^g$ for $H^1(S_g;\mathbb{Q})$ such that $f^*(a_i)=0$ for any $i\neq 1$ and $f^*(b_i)=0$ for any $i$. Let $f^*(a_1)=(x_1,x_2,...,x_n)\neq 0\in \bigoplus_{i=1}^{n}H_i\cong H^1(\text{PConf}_n(S_g);\mathbb{Q})$. Assume without loss of generality that $x_1\neq 0$. Therefore for $k\neq 1$ by Lemma \ref{2}, we have that
\[
g_k^*([\triangle])=p_k^*[S_{g}]+\sum_{i=1,i\neq k}^{n}x_i\smile p_k^*b_1\in \bigoplus_{i=1}^{n}\mathbb{Q}p_i^*[S_g]\oplus \bigoplus_{i\neq j} H_i\otimes H_j\cong H^2(S_{g}^n;\mathbb{Q}). 
\]
The coordinate $x_1\otimes p_k^*b_1$ is not zero, therefore $g_k^*([\triangle])\neq 0$. This contradicts the fact that the image of $g_i$ misses $\triangle$.\\
\\
{\bf\boldmath Case 2: $f_*$ factors though the forgetful map $p_{i*}$.} Without loss of generality, we assume that $i=1$. We have that
\[
g_2^*([\triangle])=f^*[S_{g}]+p_2^*[S_{g}]+\sum_k f^*a_k \smile p_2^*b_k-f^*b_k\smile p_2^*a_k.
\]
Since Image$(f^*)\subset \text{Image}(p_{1}^*)$, we have that $g_2^*([\triangle])$ only has nonzero terms in $ \mathbb{Q}G_{12} \oplus H_1\otimes H_2$. The fact that $g_2$ misses $\triangle$ implies
\[
f^*[S_{g}]+p_2^*[S_{g}]+\sum_k f^*a_k \otimes p_2^*b_k-f^*b_k\otimes p_2^*a_k=\lambda ([\triangle_{12}])\in  \mathbb{Q}p_1^*[S_g]\oplus \mathbb{Q}p_2^*[S_g]\oplus H_1\otimes H_2.
\]
 The coefficient of $p_2^*[S_{g}]$ tells us that $\lambda=1$. Therefore we have that $f^*[S_g]=p_1^*[S_g]$ and 
\[
\sum_k (f^*a_k-p_1^*a_k) \otimes p_2^*b_k-(f^*b_k-p_1^*b_k)\otimes p_2^*a_k=0\in H_1\otimes H_2
\]
By the property of tensor product, we know that $f^*a_k-p_1^*a_k=0$ and $f^*b_k-p_1^*b_k=0$. However in this case, if we look at the map $g_1:\text{PConf}_n(S_g)\xrightarrow{(f,p_1)}S_g\times S_g$. We have that 
\[
g_1^*([\triangle])=f^*[S_{g}]+p_1^*[S_{g}]+\sum_k f^*a_k \smile p_1^*b_k-f^*b_k\smile p_1^*a_k=2p_1^*[S_g]-2g p_1^*[S_g]=(2-2g) p_1^*[S_g]\neq 0.
\]
This contradicts the fact that the image of $g_1$ misses $\triangle$.
\end{proof}

\section{Further questions}
In this section we list a few further questions. Let $m,n$ be two positive integers. Let $(x_1,...,x_n)\in \text{PConf}_n(S)$ for any manifold $S$. Let the permutation group $\Sigma_m$ acts on $\text{PConf}_{n+m}(S)$ by permuting the last $m$ points. We have the following fiber bundle:
\begin{equation}
\text{PConf}_{m}(S-\{x_1,...,x_n\})/\Sigma_m\to \text{PConf}_{n+m}(S)/\Sigma_m\xrightarrow{f_{n+m,n}(S)} \text{PConf}_n(S).
\label{FBG}
\end{equation}
Here denote by $f_{n+m,n}(S)$ the forgetful map that forgets the first $n$ points. A section of the fiber bundle (\ref{FBG}) is called a multi-section.
\begin{que}
Classify the continuous sections of the fiber bundle (\ref{FBG}) up to homotopy for $S$ a surface.
\end{que}
\begin{que}
Classify the continuous sections of the fiber bundle (\ref{FBG}) up to homotopy for any manifold $S$.
\end{que}

\bibliography{citing}{}
\vspace{5mm}
\hfill \break

California Institute of Technology 

Department of Mathematics 

Pasadena, CA 91125, USA

E-mail: chenlei1991919@gmail.com

\end{document}